\documentclass[11pt]{article}
\usepackage{latexsym,amssymb,amscd}
\usepackage{pslatex,pifont}
\usepackage{pstricks}
\usepackage{a4wide}
\usepackage{theorem}
\DeclareMathAlphabet{\mathcal}{OMS}{cmsy}{m}{n}

\def\d{\mathsf{d}}

\def\gg{\mathfrak{g}}

\def\restrict#1{\,\vrule height1.1ex width.4pt
               depth1.2ex\lower1.0ex\hbox{\scriptsize $\:#1$}}
\newcommand{\operp}{\bigcirc\kern-7pt\perp}
\newcommand{\Ran}{\mathop{\mathrm{Ran}}\nolimits}

\newcommand{\rank}{\mathop\mathrm{rank}\nolimits}

\title{Generalized Dirichlet to Neumann operator on invariant differential forms and equivariant cohomology}
\author{Qusay S.A.~Al-Zamil \& James Montaldi}
\makeatletter \@addtoreset{equation}{section}

\makeatother

\theoremheaderfont{\sffamily\bfseries\upshape}
\newtheorem{theorem}{Theorem}[section]
\newtheorem{lemma}[theorem]{Lemma}

\newtheorem{corollary}[theorem]{Corollary}

{\theorembodyfont{\normalfont}

\newtheorem{definition}[theorem]{Definition}
\newtheorem{remark}[theorem]{Remark}

}

\newenvironment{proof}%
        {\addvspace\baselineskip\noindent {\sc Proof:}\quad}%
        {\hfill \ding{114} \par\addvspace\baselineskip}
        {\addvspace\baselineskip\noindent {\sc Proof of #1:} \quad}%
        {\hfill \ding{144} \par\addvspace\baselineskip}

\begin{document}

 \maketitle

\begin{abstract}

In \cite{Belishev2}, Belishev and Sharafutdinov  consider a compact
Riemannian manifold $M$ with boundary $\partial M$. They define a
generalized Dirichlet to Neumann (DN) operator $\Lambda$ on all
forms on the boundary  and they prove that the real additive de Rham
cohomology structure of the manifold in question is completely
determined by $\Lambda$. This shows that the DN map $\Lambda$
inscribes into the list of objects of algebraic topology. In this
paper, we suppose $G$ is a torus acting by isometries on $M$. Given
$X$ in the Lie algebra of $G$ and the corresponding vector field $X_M$
on $M$, one defines Witten's inhomogeneous coboundary operator
$\d_{X_M} = \d+\iota_{X_M}$ on invariant forms on $M$. The main
purpose is to adapt Belishev and Sharafutdinov's boundary data to
invariant forms in terms of the operator $\d_{X_M}$ and its adjoint
$\delta_{X_M}$. In other words, we define an operator
$\Lambda_{X_M}$ on invariant forms on the boundary which we call the
$X_M$-DN map and using this we recover the long exact
$X_M$-cohomology sequence of the topological pair $(M,\partial M)$
from an isomorphism with the long exact sequence formed from our
boundary data. We then show that $\Lambda_{X_M}$ completely
determines the free part of the relative and absolute equivariant
cohomology groups of $M$ when the set of zeros of the corresponding
vector field $X_M$ is equal to the fixed point set $F$ for the
$G$-action. In addition, we partially determine the mixed cup
product (the ring structure) of $X_M$-cohomology groups from
$\Lambda_{X_M}$. These results explain to what extent the
equivariant topology of the manifold in question is determined by
the $X_M$-DN map $\Lambda_{X_M}$. Finally, we illustrate the
connection between Belishev and Sharafutdinov's boundary data on
$\partial F$ and ours on $\partial M$.

\medskip



\end{abstract}

\section{Introduction}
 The classical Dirichlet-to-Neumann (DN)
operator $\Lambda_{cl}:C^{\infty}(\partial M)\longrightarrow
C^{\infty}(\partial M)$ is defined by
$\Lambda_{cl}\theta=\partial\omega/\partial \nu,$ where $\omega $ is
the solution to the Dirichlet problem $$ \Delta \omega=0, \quad
\omega\mid_{\partial M}=\theta$$ and $\nu$ is the unit outer normal
to the boundary. In the scope of inverse problems of reconstructing
a manifold from the boundary measurements, the following question is
of great theoretical and applied interest \cite{Belishev2}:

\emph{ ``To what extent are the topology and geometry of $M$
determined by the Dirichlet-to-Neumann map''?}

In this paper we are interested in the topology aspect while the
geometry aspect of the above question has been studied in
\cite{Lassas} and \cite{Lionheart}.

Much effort has been made to answer the topology aspect of this
question. For instance, in the case of a two-dimensional manifold
$M$ with a connected boundary, an explicit formula is obtained which
expresses the Euler characteristic of $M$ in terms of $\Lambda_{cl}$
where the Euler characteristic completely determines the topology of
$M$ in this case \cite{Belishev1}. In the three-dimensional case
\cite{Belishev}, some formulas are obtained which express the Betti
numbers $\beta_1(M)$ and $\beta_2(M)$ in terms of $\Lambda_{cl}$ and
$\overrightarrow{\Lambda}:C^{\infty}(T(\partial M))\longrightarrow
C^{\infty}(T(\partial M)).$

 For more topological aspects, Belishev and Sharafutdinov \cite{Belishev2} prove that the real additive de Rham cohomology  of a compact, connected, oriented smooth Riemannian  manifold $M$ of dimension $n$
with boundary is completely determined by its boundary data
$(\partial M,\Lambda)$ where $\Lambda:\Omega^{k}(\partial
M)\longrightarrow \Omega^{n-k-1}(\partial M)$ is a generalization of
the classical Dirichlet-to-Neumann operator $\Lambda_{cl}$ to the
space of differential forms. More precisely, they define the DN
operator $\Lambda$ as follows \cite{Belishev2}: given $\theta \in
\Omega^{k}(\partial M)$, the boundary value problem  $$ \Delta\omega
= 0, \quad i^{*}\omega = \theta, \quad i^{*}(\delta\omega)= 0$$ is
solvable and the operator $\Lambda$ is given by the formula
$$\Lambda\theta=i^{*}(\star \d\omega) $$ where $i^*$ is the pullback
by the inclusion map $i:\partial M \hookrightarrow M$. Here $\delta$
is the formal adjoint of $\d$ relative to the $L^2$-inner product
$$\left<\alpha,\,\beta\right> = \int_M\alpha\wedge(\star\beta)$$
which is defined on $\Omega^k(M)$, and
$\star:\Omega^k\to\Omega^{n-k}$ is the Hodge star operator.

More concretely, there are two distinguished finite dimensional
subspaces of $\mathcal{H}^{k}(M)=\ker\d\cap\ker\delta,$ whose
elements are called Dirichlet and Neumann harmonic fields
respectively, namely
$$\mathcal{H}^{k}_{D}(M) = \{\lambda\in \mathcal{H}^{k}(M)\mid
i^*\lambda=0 \}, \quad
  \mathcal{H}^{k}_{N}(M) = \{\lambda\in \mathcal{H}^{k}(M)\mid i^{*}\star\lambda=0 \}
 .$$  The dimensions of these spaces are given by $$\dim
\mathcal{H}^{k}_{D}(M)=\dim\mathcal{H}^{n-k}_{N}(M)=\beta_{k}(M)$$
where $\beta_{k}(M)$ is the $k$th Betti number \cite{Schwarz} . They
prove the following theorem
\begin{theorem}[Belishev-Sharafutdinov \cite{Belishev2}]\label{B-sh them1 }
For any $0\leq k \leq n-1,$ the range of the operator $$\Lambda+
(-1)^{nk+k+n}\d\Lambda^{-1} \d:\Omega^{k}(\partial M)\longrightarrow
\Omega^{n-k-1}(\partial M)$$ is $i^{*}\mathcal{H}^{n-k-1}_{N}(M)$.

\end{theorem}
But $i^{*}\mathcal{H}^{k}_{N}(M) \cong\mathcal{H}^{k}_{N}(M)\cong
H^{k}(M)$. Hence, $(\Lambda+ (-1)^{nk+k+1}\d\Lambda^{-1} \d)
\Omega^{n-k-1}(\partial M)\cong H^{k}(M)\cong\mathcal{H}^{k}_{N}(M)
$. Using, Poincar\'{e}-Lefscetz duality, $H^{k}(M)\cong
H^{n-k}(M,\partial M)$. So the above theorem immediately implies
that the data $(\partial M,\Lambda)$ determines the absolute and
relative de Rham cohomology groups.

  In addition, they present the following theorem which gives the lower bound for the Betti numbers of the manifold $M$ through the DN
operator $\Lambda.$
\begin{theorem}[Belishev-Sharafutdinov \cite{Belishev2}]\label{B-sh them2 }
The kernel $\ker\Lambda$ contains the space
$\mathcal{E}^{k}(\partial M)$ of exact forms and $$ \dim
[{\ker\Lambda^k}/{\mathcal{E}^{k}(\partial
 M)}]\leq \min\{\beta_{k}(\partial M),\beta_{k}(M)\}$$
 where $\beta_{k}(\partial M)$ and $\beta_{k}(M)$ are the Betti
 numbers, and $\Lambda^k$ is the restriction of $\Lambda$ to $\Omega^k(\partial M)$.
\end{theorem}

But at the end of their paper, they posed the following topological
open problem:

\emph{``Can the multiplicative structure of cohomologies be
recovered from our data $(\partial M,\Lambda)$?''}.

In 2009, Shonkwiler in \cite{Clay2} gave a partial answer to the
above question. He presents a well-defined map which is
\begin{equation}\label{clay.map}
(\phi,\psi)\longmapsto\Lambda((-1)^k \phi\wedge\Lambda^{-1}\psi),
\quad \forall (\phi,\psi) \in i^{*}\mathcal{H}^{k}_{N}(M)\times
i^{*}\star\mathcal{H}^{l}_{D}(M)
\end{equation}
 and then uses it to give a partial answer to that question.
More precisely, by using the classical wedge product between the
differential forms, he considers the mixed cup product between the
absolute cohomology $ H^{k}(M,\mathbb{R})$ and the relative
cohomology $ H^{l}(M,\partial M,\mathbb{R}) $, i.e.
$$\cup: H^{k}(M,\mathbb{R})\times H^{l}(M,\partial M,\mathbb{R}) \longrightarrow H^{k+l}(M,\partial
M,\mathbb{R})$$ and then he restricts $ H^{l}(M,\partial
M,\mathbb{R}) $ to come from the boundary subspace which is defined
by DeTurck and Gluck \cite{Gluck} as the subspace of exact forms
which satisfy the Dirichlet boundary condition (i.e. $i^*$ of these
exact forms are zero)
 %$\mathcal{E}\mathcal{H}^l_{D}(M)=\{\d\alpha
%\mid \delta\d\alpha=0, \quad i^{*}(\d\alpha)=0 \}$
and then he presents the following theorem as a partial answer to
Belishev and Sharafutdinov's question:

\begin{theorem}[Shonkwiler  \cite{Clay2}]
The boundary data $(\partial M,\Lambda)$ completely determines the
mixed cup product in terms of the map (\ref{clay.map}) when the
relative cohomology class is restricted to come from the boundary
subspace.
\end{theorem}

 From another hand, in \cite{Our paper}, we consider a compact,
oriented, smooth Riemannian manifold $M$ with boundary and we
suppose $G$ is a torus acting by isometries on $M$ and denote by
$\Omega^{k}_{G}$ the $k$-forms invarient under action of $G$. Given
$X$ in the Lie algebra of $G$ and corresponding vector field $X_M$
on $M$, we consider Witten's coboundary operator $\d_{X_M} =
\d+\iota_{X_M}$. This operator is no longer homogeneous in the
degree of the smooth invariant form on $M$: if
$\omega\in\Omega^{k}_{G}$ then $\d_{X_M}\omega \in
\Omega^{k+1}_{G}\oplus\Omega^{k-1}_{G}$. Note then that
$\d_{X_M}:\Omega^{\pm}_{G}\to\Omega^{\mp}_{G}$, where
$\Omega^\pm_{G}$ is the space of invariant forms of even ($+$) or
odd ($-$) degree. Let $\delta_{X_M}$ be the adjoint of $d_{X_M}$ and
the resulting \emph{Witten-Hodge-Laplacian} is
$\Delta_{X_M}=(\d_{X_M}+\delta_{X_M})^2 = \d_{X_M}\delta_{X_M} +
\delta_{X_M}\d_{X_M}$.

%one defines Witten's inhomogeneous coboundary operator $\d_{X_M} =
%\d+\iota_{X_M}: \Omega_G^\pm \to\Omega_G^\mp$ (where $\Omega_G^\pm$
%is even/odd invariant differential forms on $M$) and its adjoint
%$\delta_{X_M}$ and the resulting \emph{Witten-Hodge-Laplacian} is
%$\Delta_{X_M}=(\d_{X_M}+\delta_{X_M})^2 = \d_{X_M}\delta_{X_M} +
%\delta_{X_M}\d_{X_M}$.

Because the forms are invariant, it is easy to see that
$\d^2_{X_M}=0$ (see \cite{Our paper} for details). In this setting,
we define two types of $X_M$-cohomology, the absolute
$X_M$-cohomology $H^{\pm}_{X_M}(M)$ and the relative
$X_M$-cohomology $H^{\pm}_{X_M}(M,\partial M)$. The first is the
cohomology of the complex $(\Omega_G,\,\d_{X_M})$, while the second
is the cohomology of the subcomplex $(\Omega_{G,D},\,\d_{X_M})$,
where $\omega\in\Omega^{\pm}_{G,D}$ if it satisfies $i^*\omega=0$
(the $D$ is for Dirichlet boundary condition). One also defines
$\Omega^{\pm}_{G,N}(M) = \left\{\alpha\in\Omega^{\pm}_{G}(M)\mid
i^*(\star\alpha)=0\right\}$ (Neumann boundary condition). Clearly,
the Hodge star provides an isomorphism
$$\star:\Omega_{G,D}^{\pm}\stackrel{\sim}{\longrightarrow}\Omega_{G,N}^{n-\pm}$$
where we write $n-\pm$ for the parity (modulo 2) resulting from
subtracting an even/odd number from $n$. Furthermore, because
$\d_{X_M}$ and $i^*$ commute, it follows that $\d_{X_M}$ preserves
Dirichlet boundary conditions while $\delta_{X_M}$ preserves Neumann
boundary conditions. Because of boundary terms, the null space of
$\Delta_{X_M}$ no longer coincides with the closed and co-closed
forms in Witten sense. Elements of $\ker\Delta_{X_M}$ are called
\emph{$X_M$-harmonic forms} while $\omega$ which satisfy
$\d_{X_M}\omega=\delta_{X_M}\omega=0$ are \emph{$X_M$-harmonic
fields} (following \cite{Our paper}); it is clear that every
$X_M$-harmonic field is an $X_M$-harmonic form, but the converse is
false. The space of $X_M$-harmonic \emph{fields} is denoted
$\mathcal{H}^\pm_{X_M}(M)$ (so $\mathcal{H}^*_{X_M}(M)\subset
\ker\Delta_{X_M}$). In fact, the space $\mathcal{H}^\pm_{X_M}(M)$ is
infinite dimensional and so is much too big to represent the
$X_M$-cohomology, hence, we restrict $\mathcal{H}^\pm_{X_M}(M)$ into
each of two finite dimensional subspaces, namely
$\mathcal{H}^{\pm}_{X_{M},D}(M)$ and
$\mathcal{H}^{\pm}_{X_{M},N}(M)$ with the obvious meanings
(Dirichlet and Neumann $X_M$-harmonic fields, respectively). There
are therefore two different candidates for $X_M$-harmonic
representatives when the boundary is present. This construction
firstly leads us to present the $X_M$-Hodge-Morrey decomposition
theorem which states that
\begin{equation}\label{X_M.H.M}
   \Omega^{\pm}_{G}(M)=\mathcal{E}^{\pm}_{X_M}(M)\oplus
\mathcal{C}^{\pm}_{X_M}(M)\oplus \mathcal{H}^{\pm}_{X_{M}}(M)
\end{equation}
where $\mathcal{E}^{\pm}_{X_M}(M)=\{\d_{X_{M}} \alpha \mid  \alpha
\in\Omega^{\mp}_{G,D}\}$ and $\mathcal{C}^{\pm}_{X_M}(M) =
\{\delta_{X_{M}} \beta \mid  \beta \in \Omega^{\mp}_{G,N}\}.$ This
decomposition is orthogonal with respect to the $L^2$-inner product
given above.

In addition, in \cite{Our paper} we present $X_M$-Friedrichs
Decomposition Theorem which states that
\begin{eqnarray}
  \mathcal{H}^{\pm}_{X_{M}}(M) &=& \mathcal{H}^{\pm}_{X_{M},D}(M)\oplus
\mathcal{H}^{\pm}_{X_{M},\mathrm{co}}(M)\label{F1} \\
  \mathcal{H}^{\pm}_{X_{M}}(M)&=&
  \mathcal{H}^{\pm}_{X_{M},N}(M)\oplus\mathcal{H}^{\pm}_{X_{M},\mathrm{ex}}(M)\label{F2}
\end{eqnarray}
where $\mathcal{H}^{\pm}_{X_{M},\mathrm{ex}}(M)=\{ \xi \in
\mathcal{H}^{\pm}_{X_{M}}(M)\mid \xi=\d_{X_{M}}\sigma\}$  and
$\mathcal{H}^{\pm}_{X_{M},\mathrm{co}}(M)=\{ \eta \in
\mathcal{H}^{\pm}_{X_{M}}(M)\mid \eta=\delta_{X_{M}}\alpha\}$. These
give the orthogonal \emph{$X_M$-Hodge-Morrey-Friedrichs} \cite{Our
paper} decompositions,

\begin{eqnarray}\label{X_M-H.M.F}
  \Omega^{\pm}_{G}(M) &=&\mathcal{E}^{\pm}_{X_M}(M)\oplus \mathcal{C}^{\pm}_{X_M}(M)\oplus
\mathcal{H}^{\pm}_{X_{M},D}(M)\oplus \mathcal{H}^{\pm}_{X_{M},\mathrm{co}}(M)\nonumber\\
  &=&\mathcal{E}^{\pm}_{X_M}(M)\oplus \mathcal{C}^{\pm}_{X_M}(M)\oplus
\mathcal{H}^{\pm}_{X_{M},N}(M)\oplus
\mathcal{H}^{\pm}_{X_{M},\mathrm{ex}}(M)
\end{eqnarray}
The two decompositions are related by the Hodge star operator. The
orthogonality of (\ref{X_M.H.M}-\ref{X_M-H.M.F}) follows from
Green's formula for $\d_{X_M}$ and $\delta_{X_M}$ of \cite{Our
paper} which states
\begin{equation}\label{eq.Green's formula}
    \langle \d_{X_{M}}\alpha,\beta\rangle=\langle
\alpha,\delta_{X_{M}}\beta\rangle+\int_{\partial {M}} i^{*} (\alpha
\wedge \star\beta)
\end{equation}
for all $\alpha ,\beta \in\Omega_{G}$.

The consequence for $X_M$-cohomology is that each class in
$H^{\pm}_{X_M}(M)$ is represented by a unique $X_M$-harmonic field
in $\mathcal{H}^\pm_{X_M,N}(M)$, and each relative class in
$H^{\pm}_{X_M}(M,\partial M)$ is represented by a unique
$X_M$-harmonic field in $\mathcal{H}^\pm_{X_M,D}(M)$. We also
elucidate the connection between the $X_M$-cohomology groups and the
relative and absolute equivariant cohomology groups.

 Our construction of the \emph{$X_M$-Hodge-Morrey-Friedrichs} decompositions
(\ref{X_M-H.M.F}) of smooth invariant differential forms gives us
insight to create boundary data which is a generalization of
Belishev and Sharafutdinov's boundary data on
$\Omega^{\pm}_G(\partial M).$

In this paper, we take a more topological approach, looking to
determine the $X_M$-cohomology groups and the free part of the
equivariant cohomology groups from the generalized boundary data. To
this end, we need first in section 2 to prove that our concrete
realizations $\mathcal{H}^{\pm}_{X_{M},N}(M) $ and $
\mathcal{H}^{\pm}_{X_{M},D}(M)$ of the absolute and relative
$X_M$-cohomology groups respectively meet only at the origin while
in section 3 we define the $X_M$-DN operator $\Lambda_{X_M}$ on $
\Omega^{\pm}_{G}(\partial M)$, the definition involves showing that
certain boundary value problems are solvable. Our definition of
$\Lambda_{X_M}$ represents a generalization of Belishev and
Sharafutdinov's DN-operator $\Lambda$ on $ \Omega^{\pm}_{G}(\partial
M)$ in the sense that when $X_M=0$, we would get
$\Lambda_{0}=\Lambda.$ Finally, in the remaining sections, we
explain to what extent the equivariant topology of the manifold in
question is determined by the $X_M$-DN map $\Lambda_{X_M}$.

%\paragraph{Acknowledgment} The first named author would like to express his gratitude to the Ministry of Higher Education and Scientific Research of Iraq  for  the financial support  for his PhD studies in Mathematics at the University of Manchester.  This work will form part of the thesis for that PhD.

\section{Main results}

 We consider a compact, connected, oriented, smooth Riemannian manifold $M$ with boundary and we
suppose $G$ is a torus acting by isometries on $M$. Given $X$ in the
Lie algebra and corresponding vector field $X_M$ on $M$, one defines
Witten's inhomogeneous coboundary operator $\d_{X_M} =
\d+\iota_{X_M}: \Omega_G^\pm \to\Omega_G^\mp$ and the resulting
$X_M$-harmonic fields and forms as described in the introduction.

 We introduce the following definitions of the
$X_M$-trace spaces
 $$i^{*}\mathcal{H}^{\pm}_{X_{M}}(M)= \{i^* \lambda \mid \lambda \in \mathcal{H}^{\pm}_{X_{M}}(M) \},\quad i^{*}\mathcal{H}^{\pm}_{X_{M},N}(M) = \{i^* \lambda_N\mid \lambda_N \in \mathcal{H}^{\pm}_{X_{M},N}(M)\}.$$
we call $i^{*}\mathcal{H}^{\pm}_{X_{M},N}(M)$ the Neumann
$X_M$-trace space.
\begin{remark}\label{remark1}
Along the boundary of $M$, any smooth differential form $\omega$ has
a natural decomposition into tangential ( $t \omega$ ) and normal(
$n \omega$ ) components. i.e.$$ \omega\mid_{\partial M}= t \omega+n
\omega$$ and the tangential component $t \omega$ is uniquely
determined by the pull-back $i^* \omega$ and it has been denoted in
a slight abuse of notation by $i^*\omega= i^* t \omega=t \omega $.
The normal and tangential components of $\omega$ are Hodge adjoint
to each other \cite{Schwarz}, i.e. $$ \star(n \omega)=t (\star
\omega)=i^* \star \omega.$$
\end{remark}

In order to prove Theorem \ref{Thm.interseactions}, we will use the
strong unique continuation theorem, due to Aronszajn
\cite{Aronszajn1}, Aronszajn, Krzywicki and Szarski
\cite{Aronszajn}. In \cite{J. Kazdan}, Kazdan writes this theorem in
terms of Laplacian operator $\Delta$ but he mentions that it is
still valid for any operator having the diagonal form $P=\Delta I +$
lower-order terms, where $I$ is the identity matrix. Hence, one can
state this theorem in terms of diagonal form operator by the
following form:
\begin{theorem}[Strong Unique Continuation Theorem \cite{J. Kazdan}]\label{strong unique}
Let $\overline{M}$ be a Riemannian manifold with Lipschitz
continuous metric, and let $\omega$ be a differential form having
first derivatives in $L^2$ that satisfies $P( \omega)=0$ where $P$
is a diagonal form operator. If $ \omega$
 has a zero of infinite order at some point in $\overline{M}$, then $
\omega$ is identically zero on $\overline{M}$.
\end{theorem}

 Now, we are ready to present our main results.

\begin{theorem}\label{Thm.interseactions}
Let $M$ be a compact, connected, oriented smooth Riemannian manifold
of dimension $n$ with boundary and with an action of a torus $G$
which acts by isometries on $M$. If an $X_M$-harmonic field $\lambda
\in \mathcal{H}^{\pm}_{X_{M}}(M)$ vanishes on the boundary $\partial
M$, then $\lambda\equiv0$, i.e.
\begin{equation}\label{intersect}
   \mathcal{H}^{\pm}_{X_{M},N}(M)\cap
\mathcal{H}^{\pm}_{X_{M},D}(M)=\{0\}
\end{equation}
\end{theorem}
\begin{proof}
Suppose $\lambda \in\mathcal{H}^{\pm}_{X_{M},N}(M)\cap
\mathcal{H}^{\pm}_{X_{M},D}(M)$, then $\lambda$ is smooth by theorem
3.4(c) of \cite{Our paper}. Since $i^* \lambda=i^* \star\lambda=0$
then remark \ref{remark1} asserts that $t \lambda=n \lambda=0$.
Hence $ \lambda\mid_{\partial M}\equiv0$ and we get that
%$(d\lambda)\mid_{\partial M}=(\delta \lambda)\mid_{\partial
%M}=
$(\iota_{X_M} \lambda)\mid_{\partial M}=0$ as well.

The proof is local so we can consider $M$ to be the upper half space
in $\mathbb{R}^{n}$ with $\partial M=\mathbb{R}^{n-1}.$ Since the
metric, the differential form $\lambda$ and the vector field $X_M$
are given in the upper half space, we can extend them from there to
all of $\mathbb{R}^{n}$ by reflection in $\partial
M=\mathbb{R}^{n-1}.$ The resulting objects are: the extended metric,
which will be Lipschitz continuous \cite{DeTurck}; we extend
$\lambda$ to all of $\mathbb{R}^{n}$ by making it odd with respect
to reflection in $\mathbb{R}^{n-1}$ and extend $X_M$ to all of
$\mathbb{R}^{n}$ by making it even with respect to reflection in
$\mathbb{R}^{n-1}$ and extended $X_M$ will be a Lipschitz continuous
vector field. But the original $\lambda$ satisfies $
\lambda\mid_{\partial M}\equiv0$ and
$\d_{X_M}\lambda=\delta_{X_M}\lambda=0$ on $\mathbb{R}^{n-1}$, hence
the extended one will be of class $C^1$ and satisfy
$\d_{X_M}\lambda=\delta_{X_M}\lambda=0$ on $\mathbb{R}^{n}$, i.e.
the extended $\lambda$ satisfies $P(\lambda)=\Delta_{X_M}\lambda=0$
on all of $\mathbb{R}^{n}$ where the operator $\Delta_{X_M}$ has
diagonal form, i.e. $P=\Delta_{X_M}= \Delta I +$ lower-order terms,
and $I$ is the identity matrix. So far, we satisfy the first
condition of theorem \ref{strong unique}.

Now, we need to satisfy the remaining hypotheses of theorem
\ref{strong unique}. Let $x=(x',x_n)=(x_1,x_2,...,x_{n-1},x_n)$ be a
coordinates chart where $ x'=(x_1,x_2,...,x_{n-1})$  is a chart on
the boundary $
\partial M$ and $x_n$ is the distance to the boundary. In these
coordinates $x_n>0$ in $M$ and $\partial M$ is locally characterized
by $x_n=0$. These coordinates are called boundary normal coordinates
and the Riemannian metric in these coordinates has the form
$\sum_{m,r=1}^{n-1} h_{m,r}(x) dx^{m} \otimes dx^{r} + dx^{n}
\otimes dx^{n}.$

Now, we consider a neighborhood of $p \in
\partial M$ where our boundary normal coordinates are well defined.
We can write $\lambda=\alpha +\beta \wedge dx_n$ where
$\alpha=\Sigma f_{I}(x) dx^{I}$, $\beta=\Sigma g_{I}(x) dx^{I}$ and
$I\subset\{1,2,...,n-1\}$. Our goal is to prove that all the partial
derivatives of the coefficients of $\lambda$ (i.e. $f_{I}(x)$ and $
g_{I}(x)$) vanish at $p \in \partial M.$ Now, $
\lambda\mid_{\partial M}\equiv0$ which implies that
$f_{I}(x',0)=g_{I}(x',0)=0$. Hence, we can apply Hadamard's lemma to
$f_{I}(x)$ and $ g_{I}(x)$ and deduce that $f_{I}(x)=x_n
\overline{f_{I}}(x)$ and $ g_{I}(x)=x_n \overline{g_{I}}(x)$ for
some smooth functions $\overline{f_{I}}(x)$ and
$\overline{g_{I}}(x).$ Moreover, these representations for
$f_{I}(x)$ and $ g_{I}(x) $ help us to conclude that all the higher
partial derivatives of $f_{I}(x)$ and $ g_{I}(x) $ with respect to
the coordinates of $x'$ (i.e. except the normal direction coordinate
$x_n$) at the point $p$ are all zero. i.e.
$$\frac{\partial ^{\mid s \mid} f_{I}(x',0)}{\partial x_{1}^{s_1}... \partial x_{n-1}^{s_{n-1}}} = \frac{\partial ^{\mid s \mid} g_{I}(x',0)}{\partial x_{1}^{s_1}... \partial x_{n-1}^{s_{n-1}}}=0, \quad \forall s_1,s_2,...,s_{n-1}=0,1,2,...$$

Therefore, we only need to prove that all the higher partial
derivatives of $f_{I}(x)$ and $ g_{I}(x) $ in the normal direction
are zero to deduce that the Taylor series of $f_{I}(x)$ and
$g_{I}(x) $ around $ x_n=0 $ are zero.

 For contradiction, suppose the Taylor series of $f_{I}(x)$ and
$g_{I}(x) $ around $ x_n=0 $ are not zero at $p \in
\partial M $ which means that there exist the largest positive
integer numbers $k $ and $j $ such that $f_{I}(x)=x^{k}_n
\widehat{f_{I}}(x)$ and $g_{I}(x)=x^{j}_n \widehat{g_{J}}(x)$ where
$\widehat{f_{I}}(x',0)\neq 0$ and $\widehat{g_{J}}(x',0)\neq 0$ for
some $I, J$. Thus, we can always write $\lambda$ in the following
form $\lambda= x^{k}_n \tau + x^{j}_n \rho \wedge dx_n $ where the
differential forms $ \tau$ and $ \rho$ do not contain $dx_n$.
Applying $d_{X_M} \lambda=0$, we get $$ 0= d_{X_M} \lambda=k
x^{k-1}_{n} dx_n \wedge \tau   +x^{k}_{n} d\tau +x^{j}_{n} d\rho
\wedge dx_n+x^{k}_{n} \iota_{X_M} \tau+x^{j}_{n} \iota_{X_M} (\rho
\wedge dx_n ).$$

Now, reducing this equation modulo $x^{k}_n$ we conclude that the
term $ x^{j}_{n} (d\rho \wedge dx_n + \iota_{X_M} (\rho \wedge dx_n
) ) \not\equiv 0$ modulo $x^{k}_n $ because the term $k x^{k-1}_{n}
dx_n \wedge \tau \not\equiv 0 $ modulo $x^{k}_n$ and as a
consequence, we infer that $k>j$.

Similarly, we can calculate $\delta_{X_M}\lambda=-(\mp)^{n}(\star d
\star\lambda+\star \iota_{X_M} \star \lambda)=0$ ( using the
Riemannian metric above). For simplicity, it is enough to calculate
$d \star\lambda+\iota_{X_M} \star \lambda=0 $ where $\star \lambda=
x^{k}_n \xi \wedge dx_n + x^{j}_n \zeta  $ such that the
differential forms $ \xi$ and $ \zeta$ do not contain $dx_n$ and
both of them should contain many of the coefficients $h_{m,r}(x)$.
Hence, we get
$$0=d \star\lambda+\iota_{X_M} \star \lambda= x^{k}_n d\xi \wedge dx_n + j x^{j-1}_n dx_n \wedge \zeta+ x^{j}_n d\zeta + x^{k}_{n}
\iota_{X_M} (\xi \wedge dx_n )+x^{j}_{n} \iota_{X_M} \zeta.$$

Reducing this equation modulo $ x^{j}_n$ and for the same reason
above but replacing $k$ by $j$, then we can infer that $k < j $, but
this is a contradiction, then there are not such largest positive
integer numbers $k$ and $j$. Hence, the Taylor series for the
coefficients $f_{I}(x)$ and $g_{I}(x) $ around $ x_n=0 $ must be
zero at $p \in
\partial M$ , i.e. $$\frac{\partial^{r}f_{I}(x',0)}{\partial
x^{r}_{n}}= \frac{\partial^{r}g_{I}(x',0)}{\partial x^{r}_{n}}=0,
\quad \forall r=0,1,2,\dots$$ It means that all the higher partial
derivatives of $f_{I}(x)$ and $ g_{I}(x) $ we have already
considered vanish at all points of the boundary $\partial M.$ Thus,
this facts are enough to show the mixed partial derivatives
including $x_n$ also vanish at the boundary. Hence, $ \lambda$ has a
zero of infinite order at $p \in\partial M$.

The remaining possibility of one of the the coefficients $f_I$ and
$g_I$ having finite order and the other infinite order in $x_n$
follows from the same argument as above.

Thus, $ \lambda $ satisfies all the hypotheses of the strong Unique
Continuation Theorem \ref{strong unique} then $\lambda$ must be zero
on all of $\mathbb{R}^{n}$. Since $M$ is assumed to be connected,
$\lambda$ must be identically zero on all of $M$, i.e.
$\lambda\equiv0.$
\end{proof}

As a consequence of Theorem \ref{Thm.interseactions}, we obtain the
following results.
\begin{corollary}\label{coro.decomposition}
\begin{equation}\label{eq.decomposition}
   \mathcal{H}^{\pm}_{X_{M}}(M)=
\mathcal{H}^{\pm}_{X_{M},\mathrm{ex}}(M) +
\mathcal{H}^{\pm}_{X_{M},\mathrm{co}}(M)
\end{equation}
where $``+``$ is not direct sum.
\end{corollary}
\begin{proof}
The $X_M$-Friedrichs Decomposition Theorem (\ref{F1} and \ref{F2})
shows that $(\mathcal{H}^{\pm}_{X_{M},D}(M) )^{\perp}\cap
\mathcal{H}^{\pm}_{X_{M}}(M)=
\mathcal{H}^{\pm}_{X_{M},\mathrm{co}}(M) $ and $
(\mathcal{H}^{\pm}_{X_{M},N}(M) )^{\perp}\cap
\mathcal{H}^{\pm}_{X_{M}}(M)=
\mathcal{H}^{\pm}_{X_{M},\mathrm{ex}}(M) $. Hence, using these facts
together with Theorem \ref{Thm.interseactions}, we conclude
eq.(\ref{eq.decomposition}.)
\end{proof}

\begin{corollary}\label{coro.2.5}
The trace map $i^*:\mathcal{H}^{\pm}_{X_{M},N}(M)\longrightarrow
i^{*}\mathcal{H}^{\pm}_{X_{M},N}(M)$ defines an isomorphism.
\end{corollary}
\begin{proof}
It is clear that $i^*$ is surjective and we can use theorem
\ref{Thm.interseactions} to prove the kernel of the linear map $i^*$
is zero (i.e. $\ker i^*=\{0\}$) which implies that $i^*$ is
injective. Thus, $i^*$ is bijection.
\end{proof}

\begin{corollary}\label{coro.2.6}
\begin{itemize}
\item[1-] The map $f:i^{*}\mathcal{H}^{\pm}_{X_{M},N}(M)\longrightarrow
H^{\pm}_{X_M}(M)$ defined by $f(i^* \lambda_N)=[\lambda_N]$ for
$\lambda_N \in \mathcal{H}^{\pm}_{X_{M},N}(M)$ is an isomorphism.
\item[2-] The map $h:i^{*}\mathcal{H}^{n-\pm}_{X_{M},N}(M)\longrightarrow H^{\pm}_{X_M}(M,\,\partial
M)$ defined by $h(i^* \lambda_N)=[\star \lambda_N]$ for $\lambda_N
\in \mathcal{H}^{n-\pm}_{X_{M},N}(M)$ is an isomorphism.
\end{itemize}
\end{corollary}

\begin{proof}
\begin{itemize}
\item[1-] $f$ is a well-defined map because $\ker i^*=\{0\}$ (corollary
\ref{coro.2.5}). Furthermore, $f$ is a bijection because there
exists a unique Neumann $X_M$-harmonic field in any absolute $X_M$-
cohomology class (Theorem 3.16(a) of \cite{Our paper}) hence part
(1) holds.

\item[2-] It follows from part (1) by using $X_{M}$-Poincar\'e-Lefschetz duality (Theorem 3.16(c) of \cite{Our
paper}).

\end{itemize}
\end{proof}

\begin{corollary}\label{coro.2.7}
$\dim(\mathcal{H}^{\pm}_{X_{M},N}(M)
)=\dim(i^{*}\mathcal{H}^{\pm}_{X_{M},N}(M) )= \dim(
H^{\pm}_{X_M}(M))=\dim(H^{n-\pm}_{X_M}(M,\partial M)).$
\end{corollary}

In fact, it is worth saying that our paper \cite{Our paper} (in
particular, the relation between the $X_M$-cohomology and
$X_M$-harmonic fields) can be used to recover most of the results in
chapter three of \cite{Schwarz} on $\Omega^{\pm}_G(M) $  but in
terms of the operators $\d_{X_M},$  $\delta_{X_M}$ and
$\Delta_{X_M}$. In this paper we will need the following theorem
which can be proved by using the $X_M$-Hodge-Morrey-Friedrichs
decompositions (\ref{X_M-H.M.F}).
%of $\Omega^{\pm}_G(M) $ to prove the following theorem.
\begin{theorem}\label{thm.2.7}
Let $M$ be a compact, oriented smooth Riemannian manifold of
dimension $n$ with boundary and with an action of a torus $G$ which
acts by isometries on $M$. Given $\chi,\rho \in \Omega^{\mp}_G(M)$
and $\psi \in \Omega^{\pm}_G(\partial M) $, the boundary value
problem

\begin{eqnarray}\label{eq.bvp1}
  d_{X_M}\omega=\chi & and & \delta_{X_M}\omega=\rho  \quad on \quad M \nonumber\\
   & i^*\omega=\psi & \quad on \quad \partial M
\end{eqnarray}
is solvable for $\omega \in \Omega^{\pm}_G(M)$ if and only if the
data obey the integrability conditions
\begin{equation}\label{eq.integ.con1}
    \delta_{X_M} \rho=0,\quad
\left<\rho,\,\kappa\right>=0, \quad \forall \kappa \in
\mathcal{H}^{\mp}_{X_{M},D}(M)
\end{equation}
and
\begin{equation}\label{eq.integ.con2}
   d_{X_M}\chi=0, \quad
i^*\chi=d_{X_M}\psi, \quad \left<\chi,\,\kappa\right> =
\int_{\partial M} \psi\wedge i^*\star\kappa, \quad \forall \kappa
\in \mathcal{H}^{\mp}_{X_{M},D}(M)
\end{equation}

 The solution of eq.(\ref{eq.bvp1}) is unique up to arbitrary Dirichlet $X_M$-
harmonic fields $\kappa \in \mathcal{H}^{\pm}_{X_{M},D}(M)$
\end{theorem}
\begin{proof}
The proof is analogous to the proof of theorem 3.2.5 of
\cite{Schwarz} but in terms of the operators $d_{X_M}$ and
$\delta_{X_M}$ .
\end{proof}

\begin{lemma}\label{lema.1}
\begin{equation}\label{lemma on boundary}
i^{*}\mathcal{H}^{\pm}_{X_{M}}(M)=\mathcal{E}^{\pm}_{X_M}(\partial
M)+ i^{*}\mathcal{H}^{\pm}_{X_{M},N}(M)
\end{equation}
where $\mathcal{E}^{\pm}_{X_M}(\partial M)=\{d_{X_{M}} \alpha \mid
\alpha \in \Omega^{\mp}_{G}(\partial M)\} $
\end{lemma}
\begin{proof}
We first prove that,
$i^{*}\mathcal{H}^{\pm}_{X_{M}}(M)\subseteq\mathcal{E}^{\pm}_{X_M}(\partial
M)+ i^{*}\mathcal{H}^{\pm}_{X_{M},N}(M) $.

Suppose $ \lambda \in \mathcal{H}^{\pm}_{X_{M}}(M)$ then the
$X_M$-Friedrichs Decomposition  theorem (\ref{F2}) implies that
$$ \lambda=\d_{X_{M}} \alpha +\lambda_N \in
\mathcal{H}^{\pm}_{X_{M},N}(M)\oplus\mathcal{H}^{\pm}_{X_{M},\mathrm{ex}}(M)$$Hence,
$$ i^{*}\lambda=\d_{X_M}i^{*}\alpha +i^{*}\lambda_N.$$
 Conversely, it is clear that $i^{*}\mathcal{H}^{\pm}_{X_{M},N}(M)\subseteq
 i^{*}\mathcal{H}^{\pm}_{X_{M}}(M)$. So, we only need to prove that $\mathcal{E}^{\pm}_{X_M}(\partial
 M) \subseteq i^{*}\mathcal{H}^{\pm}_{X_{M}}(M)$. Suppose, $\eta=d_{X_{M}}\alpha \in \mathcal{E}^{\pm}_{X_M}(\partial
 M)$ then $\eta$ satisfies
 \begin{equation}\label{eq.2.7}
 d_{X_M} \eta=0, \quad  \int_{\partial M} d_{X_{M}}\alpha \wedge i^*\star\kappa=0, \quad \forall \kappa
\in \mathcal{H}^{\mp}_{X_{M},D}(M)
 \end{equation}
Clearly, theorem \ref{thm.2.7} asserts that the condition
(\ref{eq.2.7}) is a necessary and sufficient condition for the
existence of $\lambda \in \mathcal{H}^{\pm}_{X_{M}}(M)$ such that
$\eta=i^{*}\lambda.$
\end{proof}

\begin{remark}\label{remark2}
In \cite{Our paper}, we define the spaces
$$
  \mathcal{H}^{\pm}_{X_{M},\mathrm{co}}(M) = \{ \eta \in \mathcal{H}^{\pm}_{X_{M}}(M)\mid
  \eta=\delta_{X_{M}}\alpha\}, \qquad
  \mathcal{H}^{\pm}_{X_{M},\mathrm{ex}}(M)= \{ \xi \in \mathcal{H}^{\pm}_{X_{M}}(M)\mid \xi=d_{X_{M}}\sigma\}
$$
and our proof of the $X_M$-Friedrichs Decomposition Theorem
(\ref{F1} and \ref{F2}) shows that the differential forms $ \alpha$
and $\sigma $ can be chosen to be $X_M$-closed (i.e. $
d_{X_M}\alpha=0$ ) and $X_M$-coclosed (i.e. $\delta_{X_M}\sigma=0$)
respectively and in both cases $ \alpha$ and $\sigma $ should be
$X_M$-harmonic forms (i.e. $\Delta_{X_M}\alpha=\Delta_{X_M}\sigma=0
$). This observation will be used in section 4.
\end{remark}

\section{$X_M$-DN operator }
 Before defining this operator, we first need to prove the solvability of a certain boundary value problem \textsc{bvp} (\ref{eq.bvp2}) which is shown in
 theorem~\ref{bvp}. This theorem represents the keystone to define the $X_M$-DN operator and then to exploiting a connection
between this $X_M$-DN operator and $X_M$-cohomology via the Neumann
$X_M$-trace space $i^{*}\mathcal{H}^{\pm}_{X_{M},N}(M) $.

\begin{theorem}\label{bvp}
Let $M$ be a compact, oriented smooth Riemannian manifold of
dimension $n$ with boundary and with an action of a torus $G$ which
acts by isometries on $M$. Given $\theta \in \Omega^{\pm}_G(\partial
M)$ and $ \eta \in \Omega^{\pm}_G(M),$ then the \textsc{bvp}
\begin{equation} \label{eq.bvp2}
\left\{\begin{array}{rcl}
  \Delta_{X_{M}}\omega &=& \eta \quad \textrm{on} \quad M \\
   i^{*}\omega &=& \theta \quad \textrm{on} \quad \partial M\\
    i^{*}(\delta_{X_{M}}\omega)&=& 0 \quad  \textrm{on} \quad \partial M.
\end{array}\right.
\end{equation}
is solvable for $\omega \in \Omega^{\pm}_G(M)$ if and only if
\begin{equation}\label{eq.8}
    \left<\eta,\,\kappa_D\right>=0, \quad
\forall \kappa_D \in \mathcal{H}^{\pm}_{X_{M},D}(M)
\end{equation}
The solution of \textsc{bvp} (\ref{eq.bvp2}) is unique up to an
arbitrary Dirichlet $X_M$-harmonic field
$\mathcal{H}^{\pm}_{X_{M},D}(M).$
\end{theorem}
\begin{proof}
Suppose eq.(\ref{eq.bvp2}) has a solution then one can easily show
that the condition (\ref{eq.8}) holds by using Green's formula
(\ref{eq.Green's formula}).

Now, suppose the condition $ \left<\eta,\,\kappa_D\right>=0, \quad
\forall \kappa_D \in \mathcal{H}^{\pm}_{X_{M},D}(M)$ is given (i.e.
$\eta \in \mathcal{H}^{\pm}_{X_{M},D}(M)^{\perp}$ ). Since $\theta
\in \Omega^{\pm}_G(\partial M)$, we can construct an extension
$\omega_1 \in \Omega^{\pm}_G(M)$ of the differential form $\theta
\in \Omega^{\pm}_G(\partial M)$ such that
$$i^*\omega_1=\theta, \quad
\omega_1=\delta_{X_M}\beta_{\omega_1}+\lambda_{\omega_1}\in
\mathcal{C}^{\pm}_{X_M}(M)\oplus \mathcal{H}^{\pm}_{X_{M}}(M).$$ But
$\Delta_{X_M} \omega_1= \delta_{X_M}
d_{X_M}\delta_{X_M}\beta_{\omega_1}$, then (\ref{eq.Green's
formula}) implies that $\Delta_{X_M} \omega_1 \in
\mathcal{H}^{\pm}_{X_{M},D}(M)^{\perp} $ as well. Hence,
$\eta-\Delta_{X_M} \omega_1 \in
\mathcal{H}^{\pm}_{X_{M},D}(M)^{\perp}$. We now apply proposition
3.8 of \cite{Our paper} which for smooth invariant forms states that
for each $\overline{\eta}\in\mathcal{H}^{\pm}_{X_{M},D}(M)^{\perp}$
there is a unique smooth differential form $\overline{\omega} \in
\Omega^{\pm}_{G,D}\cap \mathcal{H}^{\pm}_{X_{M},D}(M)^{\perp}$
satisfying the \textsc{bvp} (\ref{eq.bvp2}) but with $\eta
=\overline{\eta}$ and $\theta=0$. Since $\eta-\Delta_{X_M} \omega_1
\in \mathcal{H}^{\pm}_{X_{M},D}(M)^{\perp}$ is smooth, it follows
from this there is a unique smooth differential form $\omega_2 \in
\Omega^{\pm}_{G,D}\cap \mathcal{H}^{\pm}_{X_{M},D}(M)^{\perp} $
which satisfies the \textsc{bvp}
\begin{equation} \label{eq.bvp3}
\left\{\begin{array}{rcl}
  \Delta_{X_{M}}\omega_2 &=& \eta-\Delta_{X_M} \omega_1 \quad \textrm{on} \quad M \\
   i^{*}\omega_2 &=& 0   \quad \textrm{on}   \quad \partial M\\
    i^{*}(\delta_{X_{M}}\omega_2)&=& 0  \quad \textrm{on}   \quad \partial M.
\end{array}\right.
\end{equation}
Now, let $ \omega_2=\omega-\omega_1$, then the \textsc{bvp}
(\ref{eq.bvp3}) turns into the \textsc{bvp} (\ref{eq.bvp2}). Hence,
there exists a solution to the \textsc{bvp} (\ref{eq.bvp2}) which is
$ \omega=\omega_1+\omega_2$, where the uniqueness of $\omega$ is up
to an arbitrary Dirichlet $X_M$-harmonic fields.
\end{proof}
\begin{definition}[$X_M$-DN operator $\Lambda_{X_M}$]\label{def.DN operator}
Let $M$ be the manifold in question. We consider the \textsc{bvp}
(\ref{eq.bvp2}) with $\eta=0$, i.e.
\begin{equation} \label{eq.DN map bvp3}
\left\{\begin{array}{rcl}
  \Delta_{X_{M}}\omega &=& 0 \quad \textrm{on} \quad M \\
   i^{*}\omega &=& \theta \quad \textrm{on} \quad \partial M\\
    i^{*}(\delta_{X_{M}}\omega)&=& 0 \quad  \textrm{on} \quad \partial M
\end{array}\right.
\end{equation}
then the \textsc{bvp} (\ref{eq.DN map bvp3}) is solvable
 and the solution is unique up to an arbitrary
Dirichlet $X_M$-harmonic field $ \kappa_D \in
\mathcal{H}^{\pm}_{X_{M},D}(M)$ (Theorem \ref{bvp}). We can
therefore define the $X_M$-DN operator
 $\Lambda_{X_M}:\Omega^{\pm}_{G}(\partial
M)\longrightarrow \Omega^{n-(\mp)}_{G}(\partial M)$ by
$$\Lambda_{X_M} \theta=i^{*}(\star d_{X_M} \omega). $$ Note that taking $d_{X_M} \omega$
eliminates the ambiguity in the choice of the solution $\omega$
which means $\Lambda_{X_M} \theta $
 %is independent of the choice of the solution $\omega$, hence, the operator $\Lambda_{X_M}$
is well defined.
 %and the form $ \Lambda_{X_M} \theta$ is independent of the
%choice of the solution $\omega$.
\end{definition}

In the case of $X_M=0$, the definition (\ref{def.DN operator})
reduces to the definition of Belishev and Sharafutdinov's
DN-operator $\Lambda$ \cite{Belishev2}.

The remainder of our results in this section are slightly the
analogues of the results in \cite{Belishev2}.
\begin{lemma}\label{lema.2}
Let $\omega \in \Omega^{\pm}_G(M) $ be a solution to the
\textsc{bvp} (\ref{eq.DN map bvp3}) where $\theta \in
\Omega^{\pm}_G(\partial M)$ is given. Then $d_{X_M}\omega \in
\mathcal{H}^{\mp}_{X_{M}}(M)$ and $\delta_{X_M} \omega=0.$
\end{lemma}
\begin{proof}
Since $\d_{X_M}$ commutes with $ i^{*}$ and $\Delta_{X_{M}} $ then
the \textsc{bvp} (\ref{eq.DN map bvp3}) and $\Lambda_{X_M}
\theta=i^{*}(\star d_{X_M} \omega) $ shows that $ d_{X_M}\omega$
solves the \textsc{bvp}
$$\Delta_{X_{M}}d_{X_M}\omega=0, \quad  i^{*}(\star d^{2}_{X_M} \omega)=0, \quad i^{*}(\delta_{X_M}d_{X_M} \omega)=0.$$
But proposition 3.2(4) of \cite{Our paper} implies that
$d_{X_M}\omega \in \mathcal{H}^{\mp}_{X_{M}}(M)$.

Since $d_{X_M}\omega \in \mathcal{H}^{\mp}_{X_{M}}(M)$, one can
easily verify that
$\d_{X_M}\delta_{X_M}\omega=-\delta_{X_M}\d_{X_M}\omega=0$ and
$\delta^{2}_{X_M} \omega=0$ which means that $\delta_{X_M}\omega
\in\mathcal{H}^{\pm}_{X_{M},\mathrm{co}}(M) $ but the second
condition (i.e. $i^{*}(\delta_{X_{M}}\omega)= 0 $ ) of the
\textsc{bvp} (\ref{eq.DN map bvp3}) gives that $\delta_{X_M} \omega
\in \mathcal{H}^{\pm}_{X_{M},D}(M)  $. Using (\ref{F1}), this then
implies that $\delta_{X_M} \omega \in
\mathcal{H}^{\pm}_{X_{M},D}(M)\cap\mathcal{H}^{\pm}_{X_{M},\mathrm{co}}(M)
=\{0\},$ i.e. $\delta_{X_M} \omega=0.$
\end{proof}

\begin{lemma}\label{lema.3}
The operator $\Lambda_{X_M}$ is nonnegative in the sense that the
integral $$ \int_{\partial M} \theta \wedge \Lambda_{X_M}\theta
$$ is nonnegative for any $\theta\in\Omega^{\pm}_G(\partial M)$.
\end{lemma}
\begin{proof}
For given $\theta$, let $\omega \in \Omega^{\pm}_G(M) $ be a
solution to the \textsc{bvp} (\ref{eq.DN map bvp3}). Then it follows
from (\ref{eq.Green's formula}) that
$$0=\left<\Delta_{X_M}\omega,\,\omega\right> =
\left<\d_{X_M}\omega,\,\d_{X_M}\omega\right>
+\left<\delta_{X_M}\omega,\,\delta_{X_M}\omega\right>
 - \int_{\partial
M}i^{*}\omega \wedge i^{*} (\star \d_{X_M}\omega) $$ whence
\begin{equation}\label{eq.nonnegative}
    \int_{\partial M} \theta \wedge
\Lambda_{X_M}\theta = \|\d_{X_M}\omega\|^2 +\|\delta_{X_M}\omega\|^2
\geq 0.
\end{equation}
\end{proof}

\begin{lemma}\label{lema.4}
$$\ker \Lambda_{X_M}=\Ran \Lambda_{X_M}=
i^{*}\mathcal{H}_{X_{M}}(M)$$ where
$\mathcal{H}_{X_{M}}=\mathcal{H}^{+}_{X_{M}}\oplus\mathcal{H}^{-}_{X_{M}}$
\end{lemma}
\begin{proof}
We first prove that $\ker
\Lambda_{X_M}=i^{*}\mathcal{H}_{X_{M}}(M).$ If $\theta=i^{*}\lambda
\in i^{*}\mathcal{H}_{X_{M}}(M)$ for $\lambda \in
\mathcal{H}_{X_{M}}(M) $, then $\lambda$ is a solution to the
\textsc{bvp} (\ref{eq.DN map bvp3}). But $d_{X_M}
\lambda=\delta_{X_M}\lambda=0,$ therefore
$\Lambda_{X_M}\theta=i^{*}(\star d_{X_M} \lambda)=0$. Conversely, if
$\theta \in \ker \Lambda_{X_M}$ and $\lambda$ is a solution to the
\textsc{bvp} (\ref{eq.DN map bvp3}) then $\theta=i^{*}\lambda$ and
equation (\ref{eq.nonnegative}) implies that $d_{X_M}
\lambda=\delta_{X_M}\lambda=0.$ i.e. $\theta=i^{*}\lambda\in
i^{*}\mathcal{H}_{X_{M}}(M).$ Hence, $\ker
\Lambda_{X_M}=i^{*}\mathcal{H}_{X_{M}}(M).$

Now, to prove $ \Ran \Lambda_{X_M}= i^{*}\mathcal{H}_{X_{M}}(M),$
suppose $\phi \in \Ran \Lambda_{X_M}$ then
$\phi=\Lambda_{X_M}\theta$ where $\theta=i^{*}\lambda$ such that
$\lambda$ is a solution of the \textsc{bvp} (\ref{eq.DN map bvp3}).
But, $d_{X_M}\lambda\in \mathcal{H}_{X_{M}}(M)$ (Lemma \ref{lema.2})
then $\star d_{X_M}\lambda \in \mathcal{H}_{X_{M}}(M)$ too. Hence,
$\phi=\Lambda_{X_M}\theta=i^{*}(\star d_{X_M} \lambda) \in
i^{*}\mathcal{H}_{X_{M}}(M).$ Conversely, let $\phi=i^{*}\lambda \in
i^{*}\mathcal{H}_{X_{M}}(M),$ i.e. $\lambda\in
\mathcal{H}_{X_{M}}(M)$.  Applying, the $X_M$-Friedrichs
Decomposition Theorem (\ref{F2}), we can decompose $\star\lambda$ as
\begin{equation}\label{eq.3.7}
\star\lambda=\d_{X_M} \omega+\lambda_{N} \in
\mathcal{H}_{X_{M},N}(M)\oplus\mathcal{H}_{X_{M},\mathrm{ex}}(M).
\end{equation}
Remark \ref{remark2} asserts that $\omega$ can be chosen such that
$$\Delta_{X_M}\omega=0,\quad   \delta_{X_M}\omega=0$$ which implies that
$$\Lambda_{X_M}i^{*}\omega=i^{*}(\star d_{X_M} \omega).$$ We can
obtain from eq. (\ref{eq.3.7}) that $$ i^{*}(\star d_{X_M}
\omega)=\pm i^{*}\lambda.$$ Comparing the last two equation with $
\phi=i^{*}\lambda $, we obtain $\phi=\Lambda_{X_M}(\pm i^{*}\omega)
\in \Ran\Lambda_{X_M}. $

\end{proof}

\begin{corollary}
The operator $\Lambda_{X_M}$ satisfies the following relations:
\begin{equation}\label{eq.3.6}
\Lambda_{X_M} \d_{X_M}=0, \quad  \d_{X_M}\Lambda_{X_M}=0, \quad
\Lambda^{2}_{X_M}=0.
\end{equation}
\end{corollary}
\begin{proof}
The first relation of (\ref{eq.3.6}) means that any form in the
space $\mathcal{E}_{X_M}(\partial M)$ is the trace of an
$X_M$-harmonic field which is true by $\mathcal{E}_{X_M}(\partial M)
\subseteq i^{*}\mathcal{H}_{X_{M}}(M)=\ker \Lambda_{X_M}$ (Lemmas
\ref{lema.1} and \ref{lema.4}) while the second and third of
equalities (\ref{eq.3.6}) follow from Lemma \ref{lema.4}.
\end{proof}

\begin{corollary}\label{coro. Hilbert trans.}
The operator
$\d_{X_M}\Lambda^{-1}_{X_M}:i^{*}\mathcal{H}_{X_{M}}(M)\longrightarrow
i^{*}\mathcal{H}_{X_{M}}(M)$ is well-defined, i.e. the equation
$\phi=\Lambda_{X_M}\theta$ has a solution $\theta$ for any $\phi \in
i^{*}\mathcal{H}_{X_{M}}(M)$, and $\d_{X_M}\theta$ is uniquely
determined by $\phi=\Lambda_{X_M}\theta.$ In particular, the
operator $\d_{X_M}\Lambda^{-1}_{X_M}\d_{X_M}:\Omega_{G}(\partial
M)\longrightarrow \Omega_{G}(\partial M) $ is well-defined.
\end{corollary}
\begin{proof}
Lemma \ref{lema.4}  proves that $\Ran \Lambda_{X_M}=
i^{*}\mathcal{H}_{X_{M}}(M)$. Hence, if $\phi \in
i^{*}\mathcal{H}_{X_{M}}(M)$ then the equation
$\phi=\Lambda_{X_M}\theta$ is solvable. If
$\Lambda_{X_M}\theta_1=\Lambda_{X_M}\theta_2 $ then
$\theta_1-\theta_2 \in \ker\Lambda_{X_M}$ is $X_M$-closed (i.e.
$\d_{X_M}(\theta_1-\theta_2)=0$) because
$\ker\Lambda_{X_M}=i^{*}\mathcal{H}_{X_{M}}(M) $. Thus,
$\d_{X_M}\theta_1=\d_{X_M}\theta_2$ which means that
$\d_{X_M}\theta$ is uniquely determined by
$\phi=\Lambda_{X_M}\theta.$ Clearly, the operator
$\d_{X_M}\Lambda^{-1}_{X_M}\d_{X_M}$ is well-defined because we have
shown in lemma  \ref{lema.1}  that $\mathcal{E}_{X_M}(\partial
 M) \subseteq i^{*}\mathcal{H}_{X_{M}}(M)$.
\end{proof}

\section{$\Lambda_{X_M}$ operator, $X_M$-cohomology and equivariant cohomology}
In the following theorem which is the analogues of theorem \ref{B-sh
them2 }, we relate the $\dim (H^{\pm}_{X_M}(M))$ with the kernel of
$\Lambda_{X_M}$ as follows:
\begin{theorem}\label{thm.low bounds}
Let $\Lambda^{\pm}_{X_M}$ be the restriction of $X_M$-DN operator to
$\Omega^{\pm}_{G}(\partial M)$. Then
$\mathcal{E}^{\pm}_{X_M}(\partial
 M) \subseteq \ker \Lambda^{\pm}_{X_M}$ and $$ \dim [{\ker\Lambda^{\pm}_{X_M}}/{\mathcal{E}^{\pm}_{X_M}(\partial
 M)}]\leq \min\{\dim (H^{\pm}_{X_M}(\partial M)),\dim (H^{\pm}_{X_M}(M))\}$$

\end{theorem}
\begin{proof}
We can apply the $X_M$-Hodge-Morrey decomposition theorem
(\ref{X_M.H.M}) (or theorem 2.5 of \cite{Our paper}) for $\partial
M$ which asserts that the direct sum of the first and third
subspaces is equal to the subspace of all $X_M$-closed invariant
differential $\pm$-forms (that is, $\ker\d_{X_M}$). Hence, this fact
together with eq.(\ref{eq.3.6}) implies that
$$\mathcal{E}^{\pm}_{X_M}(\partial
 M) \subset \ker\Lambda^{\pm}_{X_M}\subset \mathcal{H}_{X_M}^\pm (\partial M)
\oplus \mathcal{E}^{\pm}_{X_M}(\partial
 M).$$ This implies $$\dim [{\ker\Lambda^{\pm}_{X_M}}/{\mathcal{E}^{\pm}_{X_M}(\partial
 M)}]\leq \dim \mathcal{H}_{X_M}^\pm (\partial M)=\dim (H^{\pm}_{X_M}(\partial M)).
 $$ By Lemmas \ref{lema.1} and \ref{lema.4}, $$ \ker\Lambda^{\pm}_{X_M}= \mathcal{E}^{\pm}_{X_M}(\partial
 M)+i^{*}\mathcal{H}^{\pm}_{X_{M},N}(M).$$ Thus, $$\dim [{\ker\Lambda^{\pm}_{X_M}}/{\mathcal{E}^{\pm}_{X_M}(\partial
 M)}]\leq \dim (i^*\mathcal{H}^{\pm}_{X_{M},N}(M))=\dim (H^{\pm}_{X_M}(M)). $$
 Therefore $$ \dim [{\ker\Lambda^{\pm}_{X_M}}/{\mathcal{E}^{\pm}_{X_M}(\partial
 M)}]\leq \min \{\dim (H^{\pm}_{X_M}(\partial M)),\dim (H^{\pm}_{X_M}(M)) \}$$ as
 required.
\end{proof}

In particular, corollary 4.4 of \cite{Our paper} asserts that if the
set of zeros $N(X_M)$ of the corresponding vector field $X_M$ is
equal to the fixed point set $F$ for the $G$-action (i.e.
$N(X_M)=F$) then $\dim (H^{\pm}_{X_M}(M))=\rank H_G^\pm(M) $ and $
\dim (H^{\pm}_{X_M}(M,\partial M))=\rank H_G^\pm(M,\partial M)$
where $H_G^\pm(M) $ and $H_G^\pm(M,\partial M)$ are absolute and
relative equivariant cohomology respectively. The
$X_{M}$-Poincar\'e-Lefschetz duality (Theorem 3.16(c) of \cite{Our
paper}) asserts that $\rank H_G^\pm(M)= \rank
H_G^{n-(\pm)}(M,\partial M)$. Hence, we conclude the following
corollary which relates the kernel of $\Lambda_{X_M}$ with the rank
of the absolute and relative equivariant cohomology. In fact, we can
write down some lower bounds for that rank:

\begin{corollary}\label{coro.equivariant cohomology}
If $N(X_M)=F$ then we have $$ \dim
[{\ker\Lambda^{\pm}_{X_M}}/{\mathcal{E}^{\pm}_{X_M}(\partial
 M)}]\leq min\{\rank H_G^\pm(\partial M),\rank H_G^\pm(M) \}.$$
\end{corollary}

The following theorem is the analogues of theorem 4.2 of
\cite{Belishev2} (our theorem \ref{B-sh them1 }).
\begin{theorem}\label{thm.main result}
The Neumann $X_M$-trace space
$i^{*}\mathcal{H}^{n-(\mp)}_{X_{M},N}(M) $ can be completely
determined from our boundary data $(\partial M,\Lambda_{X_M})$ in
particular,
\begin{equation}\label{eq.main result}
    (\Lambda_{X_M}-
(\mp1)^{n+1}\d_{X_M}\Lambda^{-1}_{X_M} \d_{X_M})
\Omega^{\pm}_{G}(\partial M)=
i^{*}\mathcal{H}^{n-(\mp)}_{X_{M},N}(M)
\end{equation}

\end{theorem}

\begin{proof}
We need first to prove that $$ (\Lambda_{X_M}-
(\mp1)^{n+1}\d_{X_M}\Lambda^{-1}_{X_M} \d_{X_M})
\Omega^{\pm}_{G}(\partial M)\subseteq
i^{*}\mathcal{H}^{n-(\mp)}_{X_{M},N}(M)$$ Suppose $\theta \in
\Omega^{\pm}_{G}(\partial M)$, let $\omega \in \Omega^{\pm}_{G}(M)$
be a solution to the \textsc{bvp} (\ref{eq.DN map bvp3}). Lemma
(\ref{lema.2}) proves that $d_{X_M}\omega \in
\mathcal{H}^{\mp}_{X_{M}}(M)$. Applying the $X_M$-Friedrichs
decomposition to $\d_{X_M}\omega $, we get
\begin{equation}\label{eq.4.2}
    \d_{X_M}\omega=\delta_{X_M}\alpha+\lambda_{D} \in
    \mathcal{H}^{\mp}_{X_{M},\mathrm{co}}(M) \oplus \mathcal{H}^{\mp}_{X_{M},D}(M)
\end{equation}
where $\alpha \in \Omega^{\pm}_{G}(M) $ and by remark \ref{remark2},
$\alpha$ can be chosen such that
\begin{equation}\label{eq.4.3}
    \d_{X_M}\alpha=0, \quad \Delta_{X_M}\alpha=0
\end{equation}
we set  $\beta=\star\alpha \in \Omega^{n-\pm}_{G}(M)$. Hence,
eq.(\ref{eq.4.3}) implies
\begin{equation}\label{eq.4.4}
    \delta_{X_M}\beta=0, \quad \Delta_{X_M}\beta=0
\end{equation}
substituting $\alpha=(\pm1)^{n+1} \star \beta$ into
eq.(\ref{eq.4.2}), we have
\begin{equation}\label{eq.4.5}
\d_{X_M}\omega=(\pm1)^{n+1}\delta_{X_M}\star \beta+\lambda_{D}
\end{equation}
which implies
\begin{equation}\label{eq.4.6}
    i^{*} (\d_{X_M}\omega)=(\pm1)^{n+1} i^{*}(\delta_{X_M}\star
\beta).
\end{equation}
But, $$ i^{*} (\d_{X_M}\omega)= d_{X_M} (i^{*}
\omega)=\d_{X_M}\theta$$ and $$\delta_{X_M} \star\beta=\mp(-1)^{n}
\star \d_{X_M}\beta$$ thus, eq.(\ref{eq.4.6}) turns into
\begin{equation}\label{eq.4.7}
\d_{X_M}\theta=-(\mp1)^{n} i^{*} ( \star \d_{X_M}\beta)
\end{equation}
Formulas (\ref{eq.4.4}) and (\ref{eq.4.7}) mean that
\begin{equation}\label{eq.4.8}
d_{X_M}\theta=-(\mp1)^{n}  \Lambda_{X_M}i^{*}\beta.
\end{equation}
Now, applying, ( $i^{*}\star$ ) to eq.(\ref{eq.4.5}) with the fact
that $\Lambda_{X_M}\theta=i^{*}(\star \d_{X_M}\omega)$, we get
\begin{equation}\label{eq.4.9}
    \Lambda_{X_M}\theta=(\pm1)^{n+1} i^{*}(\star \delta_{X_M}\star
\beta)+ i^{*} (\star \lambda_{D}).
\end{equation}
Using the relation $\star \delta_{X_M} \star \beta= (\pm1)^{n}
\d_{X_M}\beta$, then eq.(\ref{eq.4.9}) reduces to
\begin{equation}\label{eq.4.10}
\Lambda_{X_M}\theta={\pm} \d_{X_M}(i^{*}\beta)+ i^{*} (\star
\lambda_{D})
\end{equation}
we can obtain from eq.(\ref{eq.4.8}) that
$$\d_{X_M}(i^{*}\beta)=-(\mp1)^{n} d_{X_M}\Lambda^{-1}_{X_M}d_{X_M}\theta$$
Putting the latter equation in eq.(\ref{eq.4.10}), we get
$$i^{*} (\star\lambda_{D})= (\Lambda_{X_M}-
(\mp1)^{n+1}\d_{X_M}\Lambda^{-1}_{X_M} \d_{X_M})\theta.$$ Hence,
$(\Lambda_{X_M}-(\mp1)^{n+1}\d_{X_M}\Lambda^{-1}_{X_M}
\d_{X_M})\theta \in i^{*}\mathcal{H}^{n-(\mp)}_{X_{M},N}(M) $.

The next step is then to prove the converse, i.e.
$$ i^{*}\mathcal{H}^{n-(\mp)}_{X_{M},N}(M) \subseteq (\Lambda_{X_M}-
(\mp1)^{n+1}\d_{X_M}\Lambda^{-1}_{X_M} \d_{X_M})
\Omega^{\pm}_{G}(\partial M)$$ Given $\lambda_N \in
\mathcal{H}^{n-(\mp)}_{X_{M},N}(M)$, then corollary
\ref{coro.decomposition} asserts that $\lambda_N$ has the following
representation
\begin{equation}\label{eq.4.11}
    \lambda_N=\d_{X_M}\alpha+\delta_{X_M}\beta \in \mathcal{H}^{n-\mp}_{X_{M},\mathrm{ex}}(M) +
\mathcal{H}^{n-\mp}_{X_{M},\mathrm{co}}(M)
\end{equation}
and also by remark \ref{remark2}, $\alpha$ and $\beta$ can be chosen
respectively to satisfy
\begin{equation}\label{eq.4.12}
    \delta\alpha=0, \quad \Delta_{X_M}\alpha=0
\end{equation}
and
\begin{equation}\label{eq.4.13}
\d_{X_M}\beta=0, \quad \Delta_{X_M}\beta=0
\end{equation}
We set up the transformations $$\omega=-(\pm1)^{n} \star\beta, \quad
\epsilon=-(\mp1)^{n+1}\alpha$$ Then
eqs.(\ref{eq.4.12})-(\ref{eq.4.13}) turn into
\begin{equation}\label{eq.4.14}
\delta\omega=0, \quad \Delta_{X_M}\omega=0
\end{equation}
\begin{equation}\label{eq.4.15}
   \delta_{X_M}\epsilon=0, \quad \Delta_{X_M}\epsilon=0
\end{equation}
and eq.(\ref{eq.4.11}) implies
\begin{equation}\label{eq.4.16}
\lambda_N=\star\d_{X_M}\omega- (\mp1)^{n+1}\d_{X_M}\epsilon
\end{equation}
hence,
\begin{equation}\label{eq.4.17}
\star\lambda_N=-(\mp1)^{n+1}(\star\d_{X_M}\epsilon -\d_{X_M}\omega).
\end{equation}
We can define forms $\phi, \psi \in \Omega_{G}(\partial M)$ by
setting
\begin{equation}\label{eq.4.18}
    \phi=i^{*}\omega, \quad \psi=i^{*}\epsilon
\end{equation}
Restricting eq.(\ref{eq.4.16}) to the boundary and using the fact
that $i^{*} \star\d_{X_M}\omega=\Lambda_{X_M}\phi$, we obtain
\begin{equation}\label{eq.4.19}
i^{*} \lambda_N=\Lambda_{X_M}\phi- (\mp1)^{n+1}\d_{X_M}
i^{*}\epsilon
\end{equation}
Restricting eq.(\ref{eq.4.17}) to the boundary
\begin{equation}\label{eq.4.20}
i^{*}(\star\d_{X_M}\epsilon) = \d_{X_M}(i^{*}\omega)
\end{equation}
but $i^{*}(\star\d_{X_M}\epsilon)=\Lambda_{X_M}\psi$ because of
eq.(\ref{eq.4.15}) and the second of equality (\ref{eq.4.18}).
Hence, eq.(\ref{eq.4.20}) turns to
\begin{equation}\label{eq.4.21}
    \Lambda_{X_M}\psi=\d_{X_M}\phi
\end{equation}
Now, we can eliminate the form $\psi$ from eq.(\ref{eq.4.19}) and
eq.(\ref{eq.4.21}) and we can obtain that
$$i^{*}\lambda_{N}= (\Lambda_{X_M}-
(\mp1)^{n+1}\d_{X_M}\Lambda^{-1}_{X_M} \d_{X_M})\phi$$ Hence,
$i^{*}\lambda_N \in (\Lambda_{X_M}-
(\mp1)^{n+1}\d_{X_M}\Lambda^{-1}_{X_M} \d_{X_M})
\Omega^{\pm}_{G}(\partial M).$
\end{proof}

\section{$X_M$- Hilbert transform}
 In this section, we introduce the $X_M$- Hilbert transform which will be used in section 6. We begin with the following
definition.
\begin{definition}[$X_M$- Hilbert transform]
The $X_M$- Hilbert transform is the operator $$T_{X_M}=
\d_{X_M}\Lambda^{-1}_{X_M}:i^{*}\mathcal{H}^{\pm}_{X_{M}}(M)\longrightarrow
i^{*}\mathcal{H}^{n-(\pm)}_{X_{M}}(M).$$ $T_{X_M}$ is a well-defined
operator by corollary \ref{coro. Hilbert trans.} and the restriction
of $T_{X_M}$ to $X_M$-exact boundary forms $
\mathcal{E}^{\pm}_{X_M}(\partial M) \subseteq
i^{*}\mathcal{H}^{\pm}_{X_{M}}(M)$ satisfies
$$T_{X_M}:\mathcal{E}^{\pm}_{X_M}(\partial M)\longrightarrow \mathcal{E}^{n-(\pm)}_{X_M}(\partial M).$$
\end{definition}

%\begin{definition}[$X_M$- conjugate form]
%Given $\omega \in \Omega^{\pm}_{G}(M)$ and $\varepsilon
%\in\Omega^{n-(\pm)}_{G}(M)$ be two $X_M$-coclosed forms (i.e.
%$\delta_{X_M}\omega=\delta_{X_M}\varepsilon=0$). if $$
%\d_{X_M}\omega=\star\d_{X_M}\varepsilon$$ then $\varepsilon$ is
%called the $X_M$- conjugate form of $\omega$. This implies that
%$\Delta_{X_M}\omega=\Delta_{X_M}\varepsilon=0$ and
%$(-1)^{(n+1)(\mp)}\omega$ is the $X_M$- conjugate form of
%$\varepsilon$.
%
%\end{definition}
\begin{lemma}\label{lema.5}
The $X_M$- Hilbert transform maps
$i^{*}\mathcal{H}^{\pm}_{X_{M},N}(M)$ to
$i^{*}\mathcal{H}^{n-(\pm)}_{X_{M},N}(M)$.
\end{lemma}
\begin{proof}
%We only need to show that $$ T_{X_M}=
%\d_{X_M}\Lambda^{-1}_{X_M}:i^{*}\mathcal{H}^{\pm}_{X_{M},N}(M)\longrightarrow
%i^{*}\mathcal{H}^{n-(\pm)}_{X_{M},N}(M)$$
Let $\varphi \in i^{*}\mathcal{H}^{\pm}_{X_{M},N}(M)$ then theorem
\ref{thm.main result} implies that $$\varphi=(\Lambda_{X_M}-
(\pm1)^{n+1}\d_{X_M}\Lambda^{-1}_{X_M} \d_{X_M} ) \theta
$$ for some $\theta \in \Omega^{n-(\mp)}(\partial M)$. Hence, it
follows that
\begin{eqnarray}
  T_{X_M} \varphi &=&\d_{X_M}\Lambda^{-1}_{X_M} (\Lambda_{X_M}-
(\pm1)^{n+1}\d_{X_M}\Lambda^{-1}_{X_M} \d_{X_M} ) \theta \nonumber  \\
   &=&(\d_{X_M}-
(\pm1)^{n+1}\d_{X_M}\Lambda^{-1}_{X_M}\d_{X_M}\Lambda^{-1}_{X_M} \d_{X_M} ) \theta  \nonumber \\
   &=&(\Lambda_{X_M}-
(\pm1)^{n+1}\d_{X_M}\Lambda^{-1}_{X_M} \d_{X_M}
)\Lambda^{-1}_{X_M}\d_{X_M}  \theta \nonumber \\
   &=& (\Lambda_{X_M}-
(\pm1)^{n+1}\d_{X_M}\Lambda^{-1}_{X_M} \d_{X_M}
)\Lambda^{-1}_{X_M}\d_{X_M} \theta \nonumber
\end{eqnarray}
but $\Lambda^{-1}_{X_M}\d_{X_M} ( \theta) \in
\Omega^{\mp}_{G}(\partial M)$. Thus, by theorem (\ref{thm.main
result}) we find that  the right hand side of the latter formula
must belong to $ i^{*}\mathcal{H}^{n-(\pm)}_{X_{M},N}(M).$
\end{proof}

\section{Recovering $X_M$-cohomology from the boundary data $(\partial M,\Lambda_{X_M})$}
 In this section we pose two questions where in subsection
 \ref{sub.6.1}  we present our answer to the following first
question:

 \emph{``Can the additive structure of the real absolute and
relative $X_M$-cohomology be completely recovered from the boundary
data $(\partial M,\Lambda_{X_M})$?.'' } The answer is affirmative
and more precisely, we show that the data $(\partial
M,\Lambda_{X_M})$ determines the long exact sequence of
$X_M$-cohomology of the topological pair $(M,\partial M)$.

 While in subsection \ref{sub.6.2}, we present a partial answer to the following
second question:

 \emph{``Can the ring (i.e. multiplicative) structure of the real absolute and relative $X_M$-cohomology be
completely recovered from the boundary data $(\partial
M,\Lambda_{X_M})$?''}

\subsection{Recovering the additive real
$X_M$-cohomology.}\label{sub.6.1}

Since the vector field $X_M$ which we are considering is always
tangent to the boundary $\partial M$ then we can still define
$X_M$-cohomology on $\partial M$, i.e. $ H^\pm_{X_M}(\partial M)$.
Hence, from our definitions of the absolute and relative
$X_M$-cohomology \cite{Our paper}, we can set up the following exact
$X_M$-cohomology sequence of the pair $(M,\partial M)$ as follows:
\begin{equation}\label{sequence.6.1}
\begin{CD}
\dots@>\pi^{*}>>H^\pm_{X_M}(M,\,\partial
M)@>\rho^{*}>>H^\pm_{X_M}(M)@>i^{*}>>H^\pm_{X_M}(\partial
M)@>\pi^{*}>>H^\mp_{X_M}(M,\,\partial M)@>\rho^{*}>>\dots
\end{CD}
\end{equation}
where
\begin{enumerate}
\item $i^{*}[\omega]_{(X_{M},M)}=[i^*\omega]_{(X_{M},\partial M)},$ \quad
$\forall [\omega]_{(X_{M},M)} \in H^\pm_{X_M}(M).$
\item $ \rho^{*}[\omega]_{(X_{M},M,\partial M)}=[\omega]_{(X_{M},M)}, \quad \forall [\omega]_{(X_{M},M,\partial M)} \in H^\pm_{X_M}(M,\,\partial M)$. In fact, the operator $\rho^{*} $ is induced by the embedding of pairs $\rho:(M,\emptyset)\subset (M,\partial
M)$. $\rho^{*}$ is well-defined.
%The definition is correct because $[\omega]_{(X_{M},M)}=0$ if
%$[\omega]_{(X_{M},M,\partial M)}=0.$
\item $\pi^{*}[\omega]_{(X_{M},\partial M)}= [\d_{X_M}\alpha]_{(X_{M},M,\partial M)}, \quad \forall[\omega]_{(X_{M},\partial M)} \in H^\pm_{X_M}(\partial
M)$, where $\alpha \in \Omega^{\pm}_{G}(M)$ is any extension of
$\omega \in \Omega^{\pm}_{G}(\partial M)$ to $M$, i.e.
$i^{*}\alpha=\omega$. Since $\d_{X_M}$ and $i^{*}$ commute, then
$[\d_{X_M}\alpha]_{(X_{M},M,\partial M)} \in
H^\mp_{X_M}(M,\,\partial M)$. The form $\d_{X_M}\alpha$ is certainly
$X_M$-exact, but is not in general relatively $X_M$-exact, i.e.
$i^{*}\alpha\neq 0.$
\end{enumerate}
Sequence (\ref{sequence.6.1}) is exact in the sense that at each
stage the image of the incoming homomorphism is the kernel of the
outgoing one.

Now, to answer the above first question, we use theorem
\ref{thm.main result} which shows that we can determine the space $
i^{*}\mathcal{H}^{\pm}_{X_{M},N}(M)$ from our boundary data and
corollary \ref{coro.2.6} which gives us the isomorphisms $f$ and
$h$.

So, if the boundary data $(\partial M,\Lambda_{X_M})$ is given then
we can construct the sequence
\begin{equation}\label{sequence.6.2}
\begin{CD}
\dots@>\overline{\pi}^{*}>>i^{*}\mathcal{H}^{n-(\pm)}_{X_{M},N}(M)@>\overline{\rho}^{*}>>i^{*}\mathcal{H}^{\pm}_{X_{M},N}(M)@>\overline{i}^{*}>>H^\pm_{X_M}(\partial
M)@>\overline{\pi}^{*}>>i^{*}\mathcal{H}^{n-(\mp)}_{X_{M},N}(M)@>\overline{\rho}^{*}>>\dots
\end{CD}
\end{equation}
where we define the operators of sequence
(\ref{sequence.6.2}) by the following formulas:
\begin{enumerate}
\item $\overline{i}^{*}\theta=[\theta]_{(X_{M},\partial M)},$ \quad
$\forall \theta \in i^{*}\mathcal{H}^{\pm}_{X_{M},N}.$ i.e.
$\theta=i^{*}\omega$ where $\omega \in \mathcal{H}^{\pm}_{X_{M},N}$
, then $\theta$ is $X_M$-closed because $i^*$ and $\d_{X_M}$
commute.
\item Using Lemma  \ref{lema.5}, then we set $$\overline{\rho}^{*}\theta=-(\pm1)^{n+1} T_{X_M}\theta, \quad \forall \theta \in
i^{*}\mathcal{H}^{n-(\pm)}_{X_{M},N}$$
\item Based on theorem  \ref{thm.main result}, then
$\Lambda_{X_M}\theta=(\Lambda_{X_M}-
(\mp1)^{n+1}\d_{X_M}\Lambda^{-1}_{X_M} \d_{X_M})\theta$ if $
[\theta]_{(X_{M},\partial M)} \in H^\pm_{X_M}(\partial M)$. Hence,
we set  $$\overline{\pi}^{*}[\theta]_{(X_{M},\partial
M)}=(\mp1)^{n+1} \Lambda_{X_M}\theta, \quad \forall \in
[\theta]_{(X_{M},\partial M)} \in H^\pm_{X_M}(\partial M).$$
%\item $\iota$ is the identity operator.

\end{enumerate}

More concretely, our goal is then to recover sequence
(\ref{sequence.6.1}) from sequence (\ref{sequence.6.2}). It means
that we should prove that the following diagram (\ref{diagram.6.3})
is commutative diagram.

\begin{equation}\label{diagram.6.3}
\begin{CD}
\dots@>\overline{\pi}^{*}>>i^{*}\mathcal{H}^{n-(\pm)}_{X_{M},N}(M)@>\overline{\rho}^{*}>>i^{*}\mathcal{H}^{\pm}_{X_{M},N}(M)@>\overline{i}^{*}>>H^\pm_{X_M}(\partial
M)@>\overline{\pi}^{*}>>i^{*}\mathcal{H}^{n-(\mp)}_{X_{M},N}(M)@>\overline{\rho}^{*}>>\dots \\
@. @VV h V @VV f V @VV \iota V @VV h V\\
\dots@>\pi^{*}>>H^\pm_{X_M}(M,\,\partial
M)@>\rho^{*}>>H^\pm_{X_M}(M)@>i^{*}>>H^\pm_{X_M}(\partial
M)@>\pi^{*}>>H^\mp_{X_M}(M,\,\partial M)@>\rho^{*}>>\dots
\end{CD}
\end{equation}
where $\iota$ is the identity operator. But, one can prove the commutativity of the diagram (\ref{diagram.6.3}) by a method similar to that given in
\cite{Belishev2} but in terms of our operators above.

Actually, the above construction proves that the data $(\partial
M,\Lambda_{X_M})$ recovers sequence (\ref{sequence.6.1}) of the pair
$(M,\partial M)$ up to an isomorphism (i.e. $f$ and $h$ are given in
corollary \ref{coro.2.6}) from the sequence (\ref{sequence.6.2}). We
therefore can state the following theorem.
\begin{theorem}\label{thm.recover additive}
The boundary data $(\partial M,\Lambda_{X_M})$ completely determines
the additive real absolute and relative $X_M$-cohomology structure
by showing the diagram (\ref{diagram.6.3}) is commutative and then
\begin{eqnarray}
  H^\pm_{X_M}(M)&\cong& (\Lambda_{X_M}-
(\pm1)^{n+1}\d_{X_M}\Lambda^{-1}_{X_M} \d_{X_M})
\Omega^{n-(\mp)}_{G}(\partial M) \label{eq.6.3}\\
  H^\pm_{X_M}(M,\,\partial
M) &\cong& (\Lambda_{X_M}- (\pm1)^{n+1}\d_{X_M}\Lambda^{-1}_{X_M}
\d_{X_M}) \Omega^{\mp}_{G}(\partial M)\label{eq.6.4}
\end{eqnarray}
\end{theorem}

\subsection{Recovering the ring structure of the real $X_M$-cohomology.}\label{sub.6.2}
First of all, we consider the mixed cup product $\overline{\cup}$
between the absolute and relative $X_M$-cohomology as follows:
$$\overline{\cup}: H^{\pm}_{X_M}(M)\times H^{\pm}_{X_M}(M,\partial M) \longrightarrow H^{\pm}_{X_M}(M,\partial M)$$
by setting
$$[\alpha]_{(X_M,M)} \overline{\cup} [\beta]_{(X_M,M,\partial M) }=[\alpha\wedge\beta]_{(X_M,M,\partial M) },\quad \forall
([\alpha]_{(X_M,M)},[\beta]_{(X_M,M,\partial M) }) \in
H^{\pm}_{X_M}(M)\times H^{\pm}_{X_M}(M,\partial M) $$ it is easy to
check that $\overline{\cup}$ is a well-defined map. In addition,
corollary 3.17 of \cite{Our paper} asserts that any absolute and
relative $X_M$-cohomology classes contain a unique Neumann and
Dirichlet $X_M$-harmonic field respectively. Hence, we can regard
any absolute (relative) $X_M$-cohomology class as a
Neumann(Dirichlet) $X_M$-harmonic field. But $[\alpha]_{(X_M,M)}
\overline{\cup} [\beta]_{(X_M,M,\partial M)
}=[\alpha\wedge\beta]_{(X_M,M,\partial M) }$ is a relative
$X_M$-cohomology class, so there exists a unique Dirichlet
$X_M$-harmonic field $\eta \in \mathcal{H}^{\pm}_{X_M,D}(M) $ such
that $[\alpha\wedge\beta]_{(X_M,M,\partial M)
}=[\eta]_{(X_M,M,\partial M) } $, i.e.
\begin{equation}\label{eq.6.6}
    \alpha\wedge\beta=\eta+d_{X_M}\xi \in \mathcal{H}^{\pm}_{X_M,D}(M) \oplus
    \mathcal{E}^{\pm}_{X_M}(M).
 \end{equation}
But, we can get from corollary \ref{coro.2.6} that
$$H^{\pm}_{X_M}(M,\partial M)\cong H^{n-(\pm)}_{X_M}(M)\cong i^{*}\mathcal{H}^{n-(\pm)}_{X_{M},N}(M)$$
According to our illustrations above we know that an absolute
$X_M$-cohomology class $[\alpha]_{(X_M,M)} \in H^{\pm}_{X_M}(M) $
and relative $X_M$-cohomology classes $[\beta]_{(X_M,M,\partial M)
}, [\alpha\wedge\beta]_{(X_M,M,\partial M) } \in
H^{\pm}_{X_M}(M,\partial M)$ are represented by the Neumann
$X_M$-harmonic field $\alpha \in\mathcal{H}^{\pm}_{X_M,N}(M)$ and
the Dirichlet $X_M$-harmonic fields $\beta, \eta \in
\mathcal{H}^{\pm}_{X_M,D}(M) $ respectively, such that they
correspond, respectively, to forms on the boundary by setting
$$ \phi=i^{*}\alpha \in i^{*} \mathcal{H}^{\pm}_{X_M,N}(M), \quad
\psi=i^{*} \star \beta \in i^{*} \mathcal{H}^{n-(\pm)}_{X_M,N}(M) ,
\quad \vartheta=i^{*} \star \eta \in i^{*}
\mathcal{H}^{n-(\pm)}_{X_M,N}(M)$$

As alluded to before, our answer to the second question will only be
partial, in the sense that we will not consider all the classes of
the relative $X_M$-cohomology. In fact, we will just consider the
\emph{boundary subspace} (which we denote by
$BH^{\pm}_{X_M}(M,\partial M)$ ) of $H^{\pm}_{X_M}(M,\partial M)$.
We define $BH^{\pm}_{X_M}(M,\partial M)$ as follows:
$$BH^{\pm}_{X_M}(M,\partial M)=\{[\d_{X_M}\varrho] \mid \varrho \in \Omega^{\mp}_{G}(M), i^*(\d_{X_M}\varrho)=0\}$$

Actually, in sequence (\ref{sequence.6.1}), our definition of the
operator $\pi^{*}$ represents the definition of the boundary
subspace of $H^{\pm}_{X_M}(M,\partial M)$. More precisely, the image
of $H^\pm_{X_M}(\partial M)$ inside $H^\mp_{X_M}(M,\,\partial M)$
represents the natural portion to interpret as coming from the
boundary. But, we have proved that  $H^{\pm}_{X_M}(M,\,\partial M)
\cong \mathcal{H}^\pm_{X_{M},D}(M)$. Hence, on translation into the
language of $X_M$-harmonic fields, we can identify
$$BH^{\pm}_{X_M}(M,\partial M)\cong \mathcal{E}\mathcal{H}^\pm_{X_{M},D}(M)$$
where $\mathcal{E}\mathcal{H}^\pm_{X_{M},D}(M)=\{\d_{X_M}\epsilon
\in \mathcal{H}^\pm_{X_{M},D}(M) \mid \epsilon \in
\Omega^{\mp}_{G}(M) \}$. Clearly, Hodge star $\star$ gives
$$c\mathcal{E}\mathcal{H}^{n-\pm}_{X_{M},N}(M)=\star \mathcal{E}\mathcal{H}^\pm_{X_{M},D}(M)
$$ where
$c\mathcal{E}\mathcal{H}^{n-\pm}_{X_{M},N}(M)=\{\delta_{X_M}\lambda
\in \mathcal{H}^{n-\pm}_{X_{M},N}(M) \mid \lambda \in
\Omega^{n-\mp}_{G}(M) \}$. Now, using this fact together with
corollary \ref{coro.2.6}(2) we conclude that
$BH^{\pm}_{X_M}(M,\partial M)\cong i^* \star
\mathcal{E}\mathcal{H}^\pm_{X_{M},D}(M)$.

Now, we adapt Shonkwiler's map \cite{Clay2} but in terms of our
operators in order to define the following map with notation as
above
\begin{equation}
\phi  \overline{\cup}_{X_M} \psi=\Lambda_{X_M}(\pm
\phi\wedge\Lambda_{X_M}^{-1}\psi), \quad \forall (\phi,\psi) \in
i^{*}\mathcal{H}^{\pm}_{X_M,N}(M)\times
i^{*}\mathcal{H}^{n-(\pm)}_{X_M,N}(M)
\end{equation}
 By using the same method as \cite{Clay2} but together with our definition \ref{def.DN operator} we deduce that
$\overline{\cup}_{X_M}$ is well-defined.

So, our partial answer to the second question is that: restricting $
H^{\pm}_{X_M}(M,\,\partial M)$ to $BH^{\pm}_{X_M}(M,\partial M)$ and
then we recover the mixed cup product by showing the commutativity
of the the following diagram.
\begin{theorem}\label{thm.comm.ring}
The diagram
\begin{equation}\label{eq.daigram 3}
   \begin{CD}
H^{\pm}_{X_M}(M)\times BH^{\pm}_{X_M}(M,\partial M)
@>\overline{\cup}>> BH^{\pm}_{X_M}(M,\partial M)\\
@VV(f,h)V @VVhV\\
i^{*}\mathcal{H}^\pm_{X_{M},N}(M)\times i^{*}
\star\mathcal{E}\mathcal{H}^{\pm}_{X_{M},D}(M)
@>\overline{{\cup}}_{X_M}>> i^{*}
\star\mathcal{E}\mathcal{H}^{\pm}_{X_{M},D}(M)
\end{CD}
\end{equation}
is commutative, where $f$ and $h$ are given in corollary
\ref{coro.2.6}.
%and can be defined as follows:
%$$f([\alpha])=i^{*}\alpha, \quad g([\d_{X_M}\beta_{1}])=i^{*}\star \d_{X_M}\beta_{1}, \quad \forall ([\alpha],[\d_{X_M}\beta_{1}])
%\in H^{\pm}_{X_M}(M)\times BH^{\pm}_{X_M}(M,\partial M).$$
\end{theorem}
\begin{proof}
Our goal is then to show that $\forall
([\alpha],[\d_{X_M}\beta_{1}]) \in H^{\pm}_{X_M}(M)\times
BH^{\pm}_{X_M}(M,\partial M)$ then
\begin{equation}\label{eq.6.7}
(h\circ
\overline{{\cup}})([\alpha],[\d_{X_M}\beta_{1}])=(\overline{\cup}_{X_M}\circ
(f,h))([\alpha],[\d_{X_M}\beta_{1}]).
\end{equation}
Using eq.(\ref{eq.6.6}), then the left-hand side gives
\begin{eqnarray}\label{eq.6.8}
  h(\overline{{\cup}}([\alpha],[\d_{X_M}\beta_{1}]))&=& h([\alpha \wedge\d_{X_M}\beta_{1}]) \nonumber\\
   &=&h([\d_{X_M}(\pm \alpha \wedge\beta_{1})]) \nonumber \\
   &=&h([\d_{X_M}(\pm \alpha \wedge\beta_{1}-\xi)] )\nonumber \\
   &=& i^{*} \star \eta
\end{eqnarray}
while the right-hand side gives
\begin{eqnarray}\label{eq.6.9}
\overline{\cup}_{X_M}((f([\alpha]) ,h( [\d_{X_M}\beta_{1}])))&=&
\overline{\cup}_{X_M}(i^{*}\alpha,i^{*}\star \d_{X_M}\beta_{1})\nonumber\\
&=&\Lambda_{X_M}(\pm \phi\wedge\Lambda_{X_M}^{-1}\psi)
\end{eqnarray}
where $\phi=i^{*}\alpha $ and $\psi=i^{*}\star \d_{X_M}\beta_{1}$.
Now, we only need to show that eq.(\ref{eq.6.8}) and
eq.(\ref{eq.6.9}) are equal.

Putting, $\beta=\d_{X_M}\beta_{1} \in
\mathcal{E}\mathcal{H}^{\pm}_{X_{M},D}(M)$ and using the
$X_M$-Hodge-Morrey decomposition theorem (\ref{X_M.H.M}), we infer that $\beta_{1}$ can be chosen to solve the \textsc{bvp} $$\Delta_{X_M}\nu=0, \quad
i^{*}\nu=i^{*}\beta_{1}, \quad i^{*}\delta_{X_M}\nu=0.
$$
Hence,$$\psi=i^{*}\star
\d_{X_M}\beta_{1}=\Lambda_{X_M}i^{*}\beta_{1}.$$ Therefore,
$\Lambda^{-1}_{X_M}\psi=i^{*}\beta_{1}$. But from eq.(\ref{eq.6.6})
we get that $$\eta=\d_{X_M}\eta' \in \mathcal{E}\mathcal{H}^{\pm}_{X_{M},D}(M)$$
where $ \eta'=\pm \alpha \wedge\beta_{1}-\xi$. Applying the $X_M$-Hodge-Morrey decomposition theorem (\ref{X_M.H.M}) on $\eta'$, we infer that $$\eta=\d_{X_M}\eta'=\d_{X_M}\sigma$$ such that $\sigma$ solves the \textsc{bvp}
$$ \Delta_{X_M} \epsilon=0,\quad i^{*} \epsilon=i^{*}\sigma, \quad i^{*}\delta_{X_M}\epsilon=0.$$
Hence,
\begin{equation}\label{eq.6.12}
    \Lambda_{X_M}i^*\sigma=i^*\star \d_{X_M}\sigma=i^* \star\eta
\end{equation}

Since $ \eta'=\pm \alpha \wedge\beta_{1}-\xi$ implies
\begin{eqnarray}\label{eq.6.13}
  \d_{X_M}(\pm  \alpha \wedge\beta_{1}) &=& \d_{X_M}\eta'+\d_{X_M}\xi \nonumber \\
   &=& \d_{X_M}\sigma+\d_{X_M}\xi.
\end{eqnarray}
Equation (\ref{eq.6.13}) shows that the class $[\pm  \alpha \wedge\beta_{1}-\sigma-\xi] \in H^{\mp}_{X_M}(M)$, so the form $\pm  \alpha \wedge\beta_{1}-\sigma-\xi$ can be decomposed as
$$\pm  \alpha \wedge\beta_{1}-\sigma-\xi=\d_{X_M}\tau_1+\tau_2 \in \mathcal{E}^{\mp}_{X_M}(M)\oplus \mathcal{H}^{\mp}_{X_{M}}(M).$$
Now, restricting the latter equation to the boundary and using Lemma \ref{lema.4}, this implies that  $$\Lambda_{X_M}i^*(\pm  \alpha \wedge\beta_{1}-\sigma-\xi)=\Lambda_{X_M}i^*\tau_2=0.$$

Combining this with equation (\ref{eq.6.12}) gives that
\begin{eqnarray}\label{eq.6.14}
 \Lambda_{X_M}i^*(\pm  \alpha \wedge\beta_{1}) &=& \Lambda_{X_M}i^*(\pm  \alpha \wedge\beta_{1}-\sigma-\xi+\sigma+\xi)\nonumber \\
   &=& \Lambda_{X_M}i^*(\pm  \alpha \wedge\beta_{1}-\sigma-\xi)+\Lambda_{X_M}i^*\sigma\nonumber \\
  \Lambda_{X_M}(\pm \phi\wedge\Lambda_{X_M}^{-1}\psi) &=& i^* \star\eta
\end{eqnarray}
Hence, the diagram (\ref{eq.daigram 3}) is commutative as desired.
\end{proof}

We can restate theorem \ref{thm.comm.ring} in the language of our
boundary data $(\partial M,\Lambda_{X_M})$ to be as follows:

\begin{theorem}\label{thm.recover ring}
The boundary data $(\partial M,\Lambda_{X_M})$ completely determines
the mixed cup product structure of the $X_M$-cohomology when the
relative $X_M$-cohomology classes comes from the boundary subspace.
i.e. if $(\alpha,\beta)\in
\mathcal{H}^{\pm}_{X_M,N}(M)\oplus\mathcal{E}\mathcal{H}^{\pm}_{X_{M},D}(M)$
such that $ \alpha\wedge\beta=\eta+d_{X_M}\xi \in
\mathcal{H}^{\pm}_{X_M,D}(M) \oplus
    \mathcal{E}^{\pm}_{X_M}(M)$ then
$$i^{*} \star\eta =\Lambda_{X_M}(\pm \phi\wedge\Lambda_{X_M}^{-1}\psi) $$
where $\phi=i^*\alpha$ and $\psi=i^*\star\beta$.
\end{theorem}

\section{Conclusions}
\begin{itemize}
\item[1-] The key which uses to recover the free part of the relative and
absolute equivariant cohomology groups (i.e. $H_{G}^\pm(M)$ and
$H_{G}^\pm(M,\partial M)$) from our boundary data $(\partial
M,\Lambda_{X_M})$ is the following theorem which has been proved in
\cite{Our paper} based on Atiyah and Bott's localization theorem:
\begin{theorem}[\cite{Our paper}]\label{relative equivariant}
Let $\{X_1,\dots,X_\ell\}$ be a basis of the Lie algebra $\gg$ and
$\{u_1,\dots,u_\ell\}$ the corresponding coordinates and let  $X =
\sum_j s_jX_j\in \gg$.  If the set of zeros $N(X_M)$ of the
corresponding vector field $X_M$ is equal to the fixed point set $F$
for the $G$-action then
\begin{equation}\label{eq.relative iso1.}
H_{X_M}^\pm(M,\,\partial M)\cong H_{G}^\pm(M,\partial
M)/\mathfrak{m}_X H_{G}^\pm(M,\partial M)\cong H^\pm(F,\partial F),
\end{equation}
and
\begin{equation}\label{eq.relative iso2.}
    H_{X_M}^\pm(M)\cong H_{G}^\pm(M)/\mathfrak{m}_X
    H_{G}^\pm(M)\cong H^\pm(F)
\end{equation}
where $\mathfrak{m}_X = \left<u_1-s_1,\dots,u_l-s_l\right>$ is the
ideal of polynomials vanishing at $X$.
\end{theorem}
Now, combining the above theorem with theorem \ref{thm.recover
additive}, we get

\begin{theorem}\label{equiv. and boundary}
With the hypotheses of the theorem \ref{relative equivariant},
$$H_{G}^\pm(M,\partial M)/\mathfrak{m}_X H_{G}^\pm(M,\partial
M)\cong H^\pm(F,\partial F)\cong (\Lambda_{X_M}-
(\pm1)^{n+1}\d_{X_M}\Lambda^{-1}_{X_M} \d_{X_M})
\Omega^{\mp}_{G}(\partial M)$$ and $$H_{G}^\pm(M)/\mathfrak{m}_X
    H_{G}^\pm(M)\cong H^\pm(F)\cong (\Lambda_{X_M}-
(\pm1)^{n+1}\d_{X_M}\Lambda^{-1}_{X_M} \d_{X_M})
\Omega^{n-(\mp)}_{G}(\partial M)$$
\end{theorem}

since the Neumann $X_M$-harmonic fields are uniquely determined by
their Neumann $X_M$- trace spaces (corollary \ref{coro.2.6}) which
can be completely determined by our boundary data $(\partial
M,\Lambda_{X_M})$(theorem \ref{thm.main result}), this means that we
can conclude from theorem \ref{thm.recover additive} and theorem
\ref{equiv. and boundary} that we can realize the relative and
absolute $X_M$-cohomology groups and the free part of the relative
and absolute equivariant cohomology groups as particular subspaces
of invariant differential forms on $\partial M$ and they are not
just determined abstractly from our boundary data.

\item[2-] If $N(X_M)=F$ then we can apply Belishev
and Sharafutdinov's results \cite{Belishev2} (our theorem \ref{B-sh
them1 }) to the manifolds $F$ with boundary $\partial F$ where $G$
acts trivially on $F$ and then we use theorem \ref{equiv. and
boundary} to exploit the connection between Belishev and
Sharafutdinov's boundary data on $\partial F$ (i.e. $(\partial
F,\Lambda)$) and ours on $\partial M$ (i.e. $(\partial
M,\Lambda_{X_M})$). More concretely, we have the following theorem
\begin{theorem}
If $N(X_M)=F$, then $$(\Lambda_{X_M}-
(\mp1)^{n+1}\d_{X_M}\Lambda^{-1}_{X_M} \d_{X_M})
\Omega^{\pm}_{G}(\partial M)\cong (\Lambda-
(\mp1)^{n+1}\d\Lambda^{-1} \d) \Omega^{\pm}(\partial F).$$
\end{theorem}

\item[3-] Theorem \ref{Thm.interseactions} proves that our concrete realizations
$\mathcal{H}^{\pm}_{X_{M},N}(M) $ and $
\mathcal{H}^{\pm}_{X_{M},D}(M)$ of the absolute and relative
$X_M$-cohomology groups inside the space $ \Omega^{\pm}_{G}(M)$ meet
only at the origin which means that we can conclude the sum
$\mathcal{H}^{\pm}_{X_{M},N}(M)+ \mathcal{H}^{\pm}_{X_{M},D}(M)$ is
a direct sum and by using (\ref{eq.Green's formula}), we can prove
that the orthogonal complement of $\mathcal{H}^{\pm}_{X_{M},N}(M)+
\mathcal{H}^{\pm}_{X_{M},D}(M)$ inside $\mathcal{H}^{\pm}_{X_{M}}(M)
$ is $$\mathcal{H}^{\pm}_{X_{M},\mathrm{ex}}(M)\cap
\mathcal{H}^{\pm}_{X_{M},\mathrm{co}}(M)=
\mathcal{H}^{\pm}_{X_{M},\mathrm{ex,co}}(M)$$ Therefore, we can
refine our $X_M$-Friedrichs Decomposition (\ref{F1} and \ref{F2})
into $$\mathcal{H}^{\pm}_{X_{M}}(M) = (
\mathcal{H}^{\pm}_{X_{M},N}(M)+ \mathcal{H}^{\pm}_{X_{M},D}(M)
)\oplus \mathcal{H}^{\pm}_{X_{M},\mathrm{ex,co}}(M). $$Consequently,
we can refine the $X_M$-Hodge-Morrey-Friedrichs decompositions
(\ref{X_M-H.M.F}) into the following five terms decomposition:
$$\Omega^{\pm}_{G}(M) =\mathcal{E}^{\pm}_{X_M}(M)\oplus
\mathcal{C}^{\pm}_{X_M}(M)\oplus ( \mathcal{H}^{\pm}_{X_{M},N}(M)+
\mathcal{H}^{\pm}_{X_{M},D}(M) )\oplus
\mathcal{H}^{\pm}_{X_{M},\mathrm{ex,co}}(M). $$The idea of this
conclusion follows from \cite{Gluck}, see also \cite{Clay1} for
details.
\end{itemize}

Finally, it is worth considering the following important open
problem:

\emph{``Can the torsion part of the absolute and relative
equivariant cohomology groups be completely recovered from our
boundary data $(\partial M,\Lambda_{X_M})$?.'' }

Answering this open problem will indeed complete the picture of our
boundary data $(\partial M,\Lambda_{X_M})$ to be adding into the
list of objects of equivariant cohomology story and consequently to
the objects of algebraic topology.

\bigskip

\parbox[t]{0.4\textwidth}{\sl School of Mathematics, \\
University of Manchester, \\
Oxford Road,\\
Manchester M13 9PL, \\
UK.}\\[12pt]

\parbox[t]{0.5\textwidth}{
\texttt{Qusay.Abdul-Aziz@postgrad.manchester.ac.uk}\\
\texttt{j.montaldi@manchester.ac.uk} }

\end{document}